\documentclass[10pt,reqno]{amsart}
\usepackage{verbatim}
\usepackage{amsthm,amsmath}
\usepackage{mathtools}
\usepackage{tabularx}
\usepackage{array}
\usepackage{enumitem}
\usepackage{longtable}
\usepackage{comment}
\usepackage{MnSymbol}
\usepackage{mathtools}
\usepackage{graphicx}    
\usepackage{xcolor}      
\usepackage{booktabs}    
\usepackage[colorlinks=true,linkcolor=blue,citecolor=blue,urlcolor=blue]{hyperref}
\usepackage[letterpaper,margin=1.25in]{geometry}   
\newcommand{\Q}{\mathbb{Q}}
\newcommand{\Qbar}{\overline{\mathbb{Q}}}
\newcommand{\Z}{\mathbb{Z}}

\newcommand{\E}{\mathcal{E}}
\renewcommand{\P}{\mathbb{P}}
\renewcommand{\O}{\mathcal{O}}
\newcommand{\Jac}{\operatorname{Jac}}
\newcommand{\End}{\operatorname{End}}
\newcommand{\Sym}{\operatorname{Sym}}
\newcommand{\Gal}{\operatorname{Gal}}

\newcommand{\grp}[1]{[#1]}

\newtheorem{theorem}{Theorem}[section]
\newtheorem{proposition}[theorem]{Proposition}
\newtheorem{corollary}[theorem]{Corollary}
\newtheorem{lemma}[theorem]{Lemma}
\theoremstyle{definition}
\newtheorem{remark}[theorem]{Remark}

\setlist[itemize]{leftmargin=18pt,itemsep=2pt,topsep=3pt}
\setlist[enumerate]{leftmargin=20pt,itemsep=2pt,topsep=3pt}

\begin{document}

\title{Rational Torsion on Simple Genus Two Jacobians}

\author{Jennifer S. Balakrishnan}
\address{Department of Mathematics and Statistics, Boston University, 665 Commonwealth Avenue, Boston, MA 02215, USA}
\email{jbala@bu.edu}

\author{Filip Najman}
\address{Department of Mathematics, Faculty of Science, University of Zagreb, Bijeni\v{c}ka cesta 30, 10000 Zagreb, Croatia}
\email{fnajman@math.hr}

\author{Ari Shnidman}
\address{Department of Mathematics, Temple University, 1805 North Broad Street, Philadelphia, PA 19122, USA}
\email{ari.shnidman@temple.edu}

\author{Andrew V. Sutherland}
\address{Department of Mathematics, Massachusetts Institute of Technology, 77 Massachusetts Avenue, Cambridge, MA 02139, USA}
\email{drew@math.mit.edu}


\date{\today}

\begin{abstract}
We exhibit new subgroups of rational torsion points in geometrically simple Jacobians of genus-two curves over $\Q$. The largest group, which has order $96$ and invariants $[2,2,2,12]$, is realized by curves of the form
\(y^2=x\prod_{t\in T}(x-t^2)\), where  $T = \{a,b,c,u,v\}$ is any set of five positive integers such that
\(a^2+b^2+c^2=u^2+v^2\) and
\(a^4+b^4+c^4=u^4+v^4\). 
We also
find realizations of the groups \(\grp{2,2,20}\), \(\grp{2,2,4,4}\),
\(\grp{2,2,2,8}\), \(\grp{2,4,8}\), and \(\grp{6,6}\).  Finally, we record, to the best of our knowledge, all known subgroups that arise in genus-two Jacobians over $\Q$, in the geometrically simple case and in general.

\end{abstract}

\maketitle

Let $X$ be a smooth projective genus-two curve over $\Q$, with affine model
\[y^2 = a_6 x^6 + a_5 x^5 + a_4 x^4 + a_3 x^3 + a_2 x^2 + a_1 x + a_0.\]
 Like an elliptic curve, $X$ admits an involution $\iota(x,y) =  (x,-y)$, but unlike an elliptic curve, the rational points of $X$ are not a group in any natural way.   To obtain a group structure, we instead pass to the Jacobian $J = \Jac(X)$, an abelian surface over $\Q$ that plays a crucial role in  the arithmetic and geometry of $X$.
 
Higher-dimensional abelian varieties are generally difficult to work with, but genus-two Jacobians are quite concrete. The non-zero elements of $J(\Q)$ are in bijection with unordered $\Gal(\bar\Q/\Q)$-stable pairs of points on $X$ {\it not} of the form $\{P, \iota(P)\}$, i.e.\ non-canonical effective divisors $D$ of degree $2$. The element $0 \in J(\Q)$ corresponds to any canonical divisor.  The group law $\oplus$ on $J(\Q)$ is defined (roughly) as follows. The inverse of $D$ is $\iota(D)$, and to add $D_1 = \sum_{i = 1}^2 (x_i,y_i)$ and $D_2 = \sum_{i = 3}^4(x_i,y_i)$, we consider the cubic polynomial $\alpha \in \Q[x]$ such that $\alpha(x_i) = y_i$ for all $i \in \{1,2,3,4\}$. The intersection $X \cap \{y = \alpha(x)\}$ consists of the four points $(x_i,y_i)$ and two others:  $D_3 = (x_5,y_5) + (x_6,y_6)$. We then define $D_1 \oplus D_2 = -D_3$, or in other words $D_1 \oplus D_2 \oplus D_3 = 0$ in $J(\Q)$.

The group \(J(\Q)_{\mathrm{tors}}\) of rational torsion points on $J$ is elementary to define, and by the Mordell--Weil theorem it is a finite abelian group. But little is known about which finite abelian groups $G$ arise in this way. 
In contrast with Mazur's theorem for elliptic curves \cite{Mazur1977}, even a conjectural
complete list of possibilities over \(\Q\) is not known in genus two, as it is generally very hard to determine whether the corresponding Siegel modular threefold with $G$-level structure has rational points. Nevertheless, searches, modular
constructions, continued fractions, and explicit descent have
produced a large collection of examples
\cite{BLP2009,DaowsudSchmidt,DaowsudSchmidtCorr,Flynn1990,Leprevost1995,LPS2004,Elkies2002,KuruSadek,Ogawa1994,PZP2013, Nicholls2018}.

We say $J$ is {\it geometrically split} if the base change $J_{\bar \Q}$ is isogenous to a
product of elliptic curves.  Howe, Lepr\'evost, and
Poonen have constructed genus-two Jacobians with large torsion subgroups by ``gluing'' elliptic curves with large torsion orders, thereby realizing groups of orders as large as \(96\)
and \(128\) \cite{HLP2000}.  In this paper, we focus on the case where \(J_{\Qbar}\) is
simple.  In other words, we look for large torsion on abelian surfaces that are unrelated to elliptic curves. While $\#J(\Q)_{\mathrm{tors}}$ is uniformly bounded in the geometrically split case, no natural number has been ruled out as a possibility for $\#J(\Q)_{\mathrm{tors}}$ in general.  

The LMFDB \cite{BSSVY,LMFDB} and the new Booker--Sutherland database \cite{BookerSutherland} provide
an empirical baseline for this question, including 65 isomorphism classes of torsion subgroups arising for geometrically simple genus-two Jacobians over~\(\Q\). The goal of this paper is to discover more groups and to organize what is known.  We write
\[
 \grp{n_1,\ldots,n_r}
   =\Z/n_1\Z\times\cdots\times\Z/n_r\Z,
 \qquad n_1\mid\cdots\mid n_r.
\]
Prior to this work, the largest documented torsion order for a geometrically simple genus-two Jacobian over \(\Q\) was \(80\), coming from Elkies' two-parameter 
\(\grp{2,2,2,10}\)-family \cite{Elkies2024}.  The largest known prime-order torsion subgroup \(\grp{31}\) comes from a recent survey carried out by Epoch AI using GPT-5.6 Sol \cite{EpochAI}, which found an example in the 2022 GitHub repository \cite{costa} corresponding to an abelian surface of $\mathrm{GL}_2$-type associated to the modular form \href{https://www.lmfdb.org/ModularForm/GL2/Q/holomorphic/1830/2/a/q/}{1830.2.a.q}, which is identified in \cite{AlessandriCoppola} as a 31-torsion candidate.  These examples are not among the 65 noted above because their realizations involve Jacobians whose conductors exceed the bounds imposed by the LMFDB and the Booker--Sutherland dataset.  This yields a total of 67 torsion subgroups known to arise for geometrically simple genus-two Jacobians over $\Q$, prior to the present work.

In this article we present realizations of six new torsion subgroups, bringing the total to 73.
Our most interesting finding is genus-two curves satisfying
\[
 J(\Q)_{\mathrm{tors}}\simeq\grp{2,2,2,12},\qquad
 \#J(\Q)_{\mathrm{tors}}=96,
\]
which is now the largest torsion subgroup known to arise for a geometrically simple genus-two Jacobian over $\Q$.
We also exhibit a new order-\(80\) torsion subgroup, \(\grp{2,2,20}\), along with the groups \(\grp{6,6}, \grp{2,4,8},\grp{2,2,2,8},\grp{2,2,4,4}\), which were not previously known to arise for geometrically simple genus-two Jacobians over $\Q$.

Table \ref{tab:census} records what we have found and organizes all groups known to be realized in geometrically simple Jacobians. At the moment,  many of the largest known groups $G$ are only known to arise on special loci within the corresponding Siegel modular threefold $A_2(G)$; see Section \ref{sec: special loci}. Are these loci simply more accessible to human search or are they fundamentally more amenable to larger torsion orders? We leave this as a guiding question going forward.    

For completeness, we also include Table \ref{tab:splitcensus}, which summarizes the currently known torsion subgroups in the geometrically split case.

The article is organized as follows.
In Section \ref{sec: basic facts} we recall relevant background on genus-two Jacobians, Mumford representation, and torsion points of small order. 
In Section \ref{sec: small and two-power}, we say more about 2-power torsion subgroups. In Section \ref{sec:symmetric}, we explain how the surface of general type defined by the equations 
\begin{align*}
    x_1^4 + x_2^4 + x_3^4 &= x_4^4 + x_5^4\\
    x_1^2 + x_2^2 + x_3^2 &= x_4^2 + x_5^2\\
    \prod_i x_i \prod_{i< j} (x^2_i &- x_j^2) \neq 0
\end{align*} 
parameterizes a natural family of genus-two Jacobians with $[2,2,2,12]$-torsion, and we exhibit infinitely many rational points on this surface. Section \ref{sec:other examples} presents some of the other new families and examples. Finally, in Tables \ref{tab:census} and \ref{tab:splitcensus}, we collect, to the best of our knowledge, all groups known to arise as $J(\Q)_{\mathrm{tors}}$ for some geometrically simple (resp.\ split) genus-two Jacobian $J$ over $\Q$.   

We hope that the results and tables in this paper will encourage others to explore rational points on the moduli spaces $A_2(G)$ and their special loci. Like modular curves, they are beautiful mathematical objects that tend to have many symmetries and remarkably simple equations relative to their geometric complexity. By studying enough examples, perhaps we will begin to understand whether and why $A_2(G)(\Q)$ should be empty for $\#G$ large enough, as is predicted.  While we expect this question to remain open for the foreseeable future, Section \ref{sec: open questions} gives many intermediate open questions that are more accessible.

\subsection*{AI use disclosure} 
Throughout this project, frontier AI models were used for mathematical exploration and triage, writing code, carrying out computations, and organizing data. Details of our workflow can be found in Section \ref{sec: ai}.  Aside from Tables~\ref{tab:census} and~\ref{tab:splitcensus}, which were formatted by AI-written code reading independently verified source data, this manuscript was written entirely by the authors.  AI tools were used for proofreading but not editing.

All mathematical results stated in this manuscript, as well as AI-generated code and computations verifying them, have been independently verified by the authors, who take full responsibility for the accuracy and content of this work.  A complete set of human-verified data and Magma \cite{Magma} source code sufficient to verify all mathematical claims not directly supported by arguments made in this manuscript is available in the GitHub repository~\cite{certification_repo}.

\subsection*{Acknowledgements}
This project started at the 2025 ICERM workshop ``Algebraic Points on Curves'', and we are grateful to the organizers of the workshop, as well as ICERM, for facilitating productive discussions. Additionally, we thank Noam Elkies, Sam Frengley, Sachi Hashimoto, Brendan Hassett, Jef Laga, Adam Logan, and Yuri Zarhin for helpful conversations.
J.S.B. was partially supported by NSF DMS-1945452 CAREER and DMS-2502687.
F.N. was funded by the Croatian Science Foundation under the project no. IP-2022-10-5008, by the  project ``Implementation of cutting-edge research and its application as part of the Scientific Center of Excellence for Quantum and Complex Systems, and Representations of Lie Algebras'', PK.1.1.10.0004, co-financed by the European Union through the European Regional Development Fund -- Competitiveness
and Cohesion Programme 2021--2027, and by the European Union – NextGenerationEU through the
National Recovery and Resilience Plan 2021–2026, via an institutional grant of the University
of Zagreb Faculty of Science, IK IA 1.1.3, Impact4Math.
A.S. was partially funded by the
European Research Council (ERC, CurveArithmetic, 101078157), as well as the Ambrose Monell
Foundation.
A.V.S. was partially supported by grants from the Simons Foundation (MPS-Infrastructure-00008651), Renaissance Philanthropy (AI for Math), and DARPA (expMath).

\section{Preliminaries on genus-two Jacobians}\label{sec: basic facts}

Let $k$ be a field of characteristic not $2$, and let $X$ be a smooth projective geometrically connected curve over $k$ of genus two. Let $K$ denote a canonical divisor. The linear system $|K|$ gives rise to a double cover $\pi \colon X \to \P^1$ branched at six points, and $X$ has a model in weighted projective space $\P(1,3,1)$ 
\[y^2 + h(x,z)y = f(x,z),\]
where $h$ is cubic, $f$ is sextic, and the hyperelliptic involution is $\iota(x:y:z) = (x:-h(x,z) - y:z)$. We may choose a model $y^2 = f'(x,z)$ with $h = 0$ by taking $f' = 4f + h^2$. For more on the arithmetic of hyperelliptic curves, see \cite{Stoll2001}. 

Fix a model $y^2 = f(x,z)$. The canonical divisors are the scheme-theoretic fibers $\pi^{-1}([x : z])$ for some $[x : z] \in \P^1$. The ramification points of $\pi$ are the six Weierstrass points $W = (x : 0 : z) \in X(\overline{k})$, where $[x : z]$ is a root of $f(x,z)$. Note that $2W$ is a canonical divisor on $X_{\bar k}$.   

\subsection{Rational points on \texorpdfstring{$J$}{\it J} and the addition law}
Consider a model $X \colon y^2 + hy = f$.  The Jacobian $J = \Jac(X)$ is an abelian surface over $k$ parameterizing linear equivalence classes of degree $0$ divisors on $X$.  A non-zero $p \in J(k)$ can be written {\it uniquely} as $[D -K]$, for an effective divisor $D$ of degree $2$. 

Thus, $J(k)\setminus \{0\}$ is in bijection with the complement of  the canonical divisors in  $\Sym^2(X)(k)$, where $ \Sym^2(X) = X^2/S_2$. We sometimes identify $p \in J(k)$  with the corresponding $D \in \Sym^2(X)(k)$.  A non-canonical $D \in \Sym^2(X)(k)$ can be described by equations of the form
\[Q(x,z) = 0, y = \alpha(x,z)\]
for a homogeneous quadratic $Q \in k[x,z]$ (unique up  to scaling) and a homogeneous cubic $\alpha \in k[x,z]$ satisfying $Q \mid \alpha^2 + h\alpha - f$ (unique up to translations by multiples of $Q$).  The pair $(Q,\alpha)$ is a {\it Mumford representation} of $D -K$.

The addition law on $J(k)$ is as follows.   The inverse of $D-K$ is $\iota(D)-K$, since $D + \iota(D) \equiv 2K$ (where $\equiv$ denotes linear equivalence). Given non-zero $D_1 = (Q_1,\alpha_1)$ and $D_2 = (Q_2, \alpha_2)$ in $J(k)$, we wish to find a Mumford representation $D = (Q,\beta)$ of $D_1+D_2-2K$. If $D_1 = P + R$ and $D_2 = \iota(P) + S$ for points $P,R,S$, then $D = R+S$,  so we will assume that this is not the case.  Then, by Riemann--Roch, there exists a unique cubic $\alpha \in k[x,z]$ such that  $\mathrm{div}(y - \alpha) + 3K \geq D_1 + D_2$, or in other words $\alpha(x_i,z_i) = y_i$ for all $P_i = [x_i : y_i : z_i]$ in the support of $D_1+D_2$.
Then $\alpha^2 + \alpha h - f$ is divisible by both $Q_1$ and $Q_2$, hence equals $Q_1Q_2Q$ for some quadratic $Q$.   The divisor of the function $y - \alpha$ is $D_1 + D_2 + D_3 - 3K$, so we may take $D = -D_3 = (Q, -h - \alpha)$.  

\subsection{Points of order 2}
The $15$ classes $D - K$ of order $2$ in $J(\overline{k})$ are the pairs $D = W_1+W_2$ of distinct Weierstrass points. Indeed, $2(W_1+W_2) \equiv 2K$ and conversely there are $\binom{6}{2} = 15$ such pairs.  In particular, we have $J[2](k) \simeq [2,2,2,2]$ if and only if $X$ has affine model $y^2 = f(x)$, where $f(x)$ splits completely over $k$. We may then move two Weierstrass points to $0$ and $\infty$, giving an affine model of the form $y^2 = x\prod_{i = 1}^4 (x+a_i)$ for $a_i \in k$.

In general, if $X$ has model $y^2 +hy= f$, then the elements of order $2$ in $J(k)$ are in bijection with the degree two factors of $4f +h^2$. 

\subsection{Points of order 3}\label{subsec: points of order 3}
Suppose $[D - K] \in J(k)$ has order $3$. Choosing $\alpha \in k[x,z]$ so that $y - \alpha$ has divisor $3D - 3K$, and then translating the model so that $\alpha = 0$, we obtain a new model $y^2 + hy = \lambda Q^3$ for some scalar $\lambda$. Completing the square, we recover the well-known fact that the Jacobian of a genus-two curve $y^2 = f(x)$ has a $k$-rational $3$-torsion point if and only if $f$ can be written as $f = h^2 - \lambda Q^3$, for a cubic $h$, a quadratic $Q$, and a scalar $\lambda$.

\subsection{Points of order 6}

It is not hard to find a three-parameter family of curves with $[6] \hookrightarrow J(k)_{\mathrm{tors}}$, but the formulas are somewhat messy.  For later use, we instead highlight the following two-parameter family:

\begin{theorem}\label{thm: M(6)}
    There exists $P \in X(k)$ and a Weierstrass point $W \in X(k)$ such that $[P - W] \in J(k)$ has order $6$ if and only if $X$ has an affine model 
    \[y^2 = (x^3 + h_1x^2 + h_2x + h_3)^2 - h_3^2,\]
    with $h_i \in k$.
    Letting $\infty_+$ and $\infty_-$ be the rational points at infinity, the divisor classes $[\infty_\pm - (0,0)]$ have order $6$, while the classes $[\pm (\infty_+ - \infty_-)]$ have order $3$.
\end{theorem}
\begin{proof}
    Since $2P + 2P + 2P \equiv 3K$, we may choose a cubic $\alpha(x)$ such that the divisor of $y - \alpha$ has multiplicity $6$ at $P$. Translating the $y$-coordinate by $\alpha$ and moving $P$ to infinity, we may assume $\alpha = 0$ and $X$ has affine model
    \[y^2 + h(x)y = \lambda,\]
    for some cubic $h(x)$ and constant $\lambda$. Completing the square and taking $t = -4\lambda$, $X$ has a model of the form
    \[y^2 = h(x)^2 - t.\]
    Since $[P - W] = [P+W - K]$ is rational, so is $W$, so $h(x)^2 - t$ has a root and in particular $t$ is a square. Moving the root to $x = 0$, we obtain the desired model.
    
    Conversely, suppose $X$ has model $y^2 = h(x)^2 - h_3^2=:f(x)$ as above. Then $f = f_1f_2f_3$ with $f_i$ of degree $i$, so $X$ has a rational Weierstrass point $(0,0)$ and $J(k)$ has a point of order $2$. Taking $Q = 1$ in Section \ref{subsec: points of order 3}, we also see that $J(k)$ has two points of order three supported at infinity, which are double the classes $[\infty_{\pm} - (0,0)]$. The latter classes have order $6$, since there are no order $3$ classes of the form $P - W$ on a genus-two curve, as this would imply $2P \equiv \iota(P) + W$. 
\end{proof}

Informally, Theorem \ref{thm: M(6)} shows that the  space $M(6)$ of triples $(X,W,P)$, where $X$ is a genus-two curve, $W \in X$ is a Weierstrass point and $P \in X$ is a point such that $P - W \in J(k)$ has order $6$, is rational. Indeed, it is birational to the weighted projective space $\P(1,2,3)$ in the parameters $h_1, h_2, h_3$.

\section{Large 2-groups}\label{sec: small and two-power}

In this section, we suppose that $J[2](k) = [2,2,2,2]$. In this case, $2$-descent supplies a homomorphism 
\[J(k)/2J(k) \stackrel{\delta}{\hookrightarrow} (k^\times/k^{\times2})^4\]  
that can be made very explicit; see \cite{Stoll2001} for further details.
\begin{lemma}\label{lem: division by 2 via descent}
    If $X/k$ has a model
    \[X \colon y^2 = x(x - a_1)(x-a_2)(x-a_3)(x-a_4)\] 
    and  $P = (x_0,y_0) \in X(k)$ with $y_0 \neq 0$, then $\delta(P - \infty) = (x_0 - a_1,x_0 - a_2, x_0 -a_3, x_0 - a_4)$. 
    If $D-2\infty \in J(k)$ has Mumford representation $(q,\alpha)$ such that $q(a_i) \neq 0$, then \[\delta(D - 2\infty) = (q(a_1), q(a_2), q(a_3), q(a_4)).\]    
\end{lemma}
\begin{proof}
    See \cite[\S4]{Stoll2001}. 
\end{proof}

We will also use the following complement of Lemma \ref{lem: division by 2 via descent}.

\begin{lemma}\label{lem: descent formula for two-torsion points}
    For $X$ as above, we have
    \begin{enumerate}
        \item $\delta((0,0) - \infty) = (-a_1,-a_2,-a_3,-a_4)$,
        \item $\delta((a_1,0) - \infty) = (a_1\prod_{j \neq 1}(a_1  -a_j), a_1 - a_2, a_1 - a_3, a_1 - a_4)$,
        \item $\delta((a_2,0) - \infty) = (a_2 - a_1, a_2\prod_{j \neq 2}(a_2  -a_j), a_2 - a_3, a_2 - a_4)$,
        \item $\delta((a_3,0) - \infty) = (a_3 - a_1, a_3 - a_2, a_3\prod_{j \neq 3}(a_3  -a_j), a_3 - a_4)$,
        \item $\delta((a_4,0) - \infty) = (a_4 - a_1,a_4 - a_2,a_4 - a_3, a_4\prod_{j \neq 4}(a_4  -a_j))$.
        \end{enumerate}
        In particular, $(0,0) - \infty \in 2J(k)$ if and only if $-a_i \in k^{\times 2}$ for all $i$, and $(a_i,0) - \infty \in 2J(k)$ if and only if $a_i \in k^{\times 2}$ and $a_i - a_j \in k^{\times2}$ for all $j \neq i$.
 
\end{lemma}

From these lemmas, we obtain explicit divisible-by-$2$ criteria, and hence explicit conditions for $4$- and $8$-torsion.
\begin{lemma}\label{lem: 2244}
Let $X$ and $J$ be as above. 

\begin{enumerate}
\item[\textup{(a)}] We have $[2,2,2,2]\subseteq J(k)$ if and only if $X$ has a model of the form
\[
X\colon y^2=x(x+A)(x+B)(x+C)(x+D),
\]
for distinct non-zero $A,B,C,D \in k$.
\item[\textup{(b)}] We have $[2,2,2,4]\subseteq J(k)$ if and only if we can choose a model in \textup{(a)} with $A = a^2,B = b^2,C = c^2,D = d^2$ squares.  Then $(0,0) - \infty \equiv 2(D_0-2\infty)$, where $D_0$ has Mumford representation 
\[(x^2 - s_2x + s_4, (s_1s_2-s_3)x - s_1s_4),\] 
and where $s_i$ is the $i$th elementary symmetric polynomial in $a,b,c,d$. 
\item[\textup{(c)}] We have $[2,2,4,4]\subseteq J(k)$ if and only if we can choose the model in \textup{(b)} so that 
\[
(C-A)(C-B),\  (C-A)(D-A),\  (C-B)(D-B),\  (D-A)(D-B)
\]
are all squares, in which case $(-A,0) - (-B,0)\in 2J(k)$.

\item[\textup{(d)}] We have $[2,2,2,8] \subset J(k)$ if and only if we can choose the model in \textup{(b)} so that 
\[a(a+d)(a+c)(a+b), b(b+d)(b+c)(a+b), c(c+d)(b+c)(a+c), d(c+d)(b+d)(a+d)\]
are all squares, in which case $[D_0-2\infty] \in 2J(k)$.

\item[\textup{(e)}] We have $[2,2,4,8]\subset J(k)$ if and only if we can choose a model in \textup{(c)} that also satisfies \textup{(d)}.
\end{enumerate}
\end{lemma}
\begin{proof}
$(a)$ is clear, while $(b)$ and $(c)$ follow from Lemma \ref{lem: descent formula for two-torsion points}. The formula for the Mumford representation in $(b)$ follows from Zarhin's general division-by-$2$ formula \cite[Example 3.7]{zarhin}. 
Part $(d)$ follows from Lemma \ref{lem: division by 2 via descent} and the formula for $q$ in $(b)$; the four expressions in $(d)$ are exactly the values $q(a_i)$.
\end{proof}

\begin{corollary}
The moduli space of genus-two curves with six marked Weierstrass points and a point of order $4$ in $J(k)$ is birational to $\P^3_{a,b,c,d}$. The curve corresponding to $[a : b : c : d]$ is $y^2 = x(x+a^2)(x+b^2)(x+c^2)(x+d^2)$.
\end{corollary}

\section{Genus-two Jacobians with \texorpdfstring{$\grp{2,2,2,12}$}{[2,2,2,12]}-torsion}
\label{sec:symmetric}

We combine some of the results from the previous sections to find a model for the universal genus-two curve with $6$ Weierstrass points $W_1,\ldots, W_6$ and a class $[D - K] \in J(k)$ of order $12$ such that $2[D-K] = [P - W_1]$ for some $P \in X(k)$.  

Let \(\mathcal S\subset\mathbf P^4\) be the complete intersection
\begin{equation}
\begin{split}
 a^2+b^2+c^2&=u^2+v^2,\\
 a^4+b^4+c^4&=u^4+v^4.
\end{split}
\label{eq:surface}
\end{equation}
Let \(\mathcal S^\circ \subset \mathcal{S}\) be the open subset where all five coordinates are
nonzero and their squares distinct.  We associate with
\(P=[a:b:c:u:v]\in\mathcal S^\circ(\Q)\) the curve 
\begin{equation}
 X_P:\quad
 y^2=x(x-a^2)(x-b^2)(x-c^2)(x-u^2)(x-v^2).
\label{eq:symmetric-curve}
\end{equation}
The open condition is exactly the nonvanishing of the discriminant, so
\(X_P\) is a smooth curve of genus two.

\begin{theorem}
\label{thm:symmetric [2,2,2,12]}
For every \(P\in\mathcal S^\circ(\Q)\), we have $\grp{2,2,2,12} \mathrel{\scalebox{1.7}[1]{$\hookrightarrow$}}
 \Jac(X_P)(\Q).$
\end{theorem}

\begin{proof}
Putting
\[
 p=a^2,\quad q=b^2,\quad r=c^2,\quad s=u^2,\quad t=v^2,
\]
the equations for \(\mathcal S\) give
\[
 p+q+r=s+t,\qquad pq+pr+qr=st.
\]
Consequently,
\begin{equation}
 (x-p)(x-q)(x-r)=x(x-s)(x-t)-pqr.
\label{eq:constant-difference}
\end{equation}
Setting $h = x^3 -(s+t)x^2 +stx -\frac12 pqr$, the curve $X_P$ is $y^2 = h(x)^2 - (\frac12pqr)^2$. 
By Theorem \ref{thm: M(6)}, the classes $[\infty_{\pm} - (0,0)]$ have order six.   

All six Weierstrass points of $X_P$ are rational, hence
\(J[2](\Q)\simeq(\Z/2)^4\).  It remains to check that $J$ has a point of order $4$ that doubles to $W_s-W_t$, where $W_i$ is the Weierstrass point above $x = i$. For this, we 
make the change of coordinates
\[
 \tilde x=\frac{x-s}{t-x}
\]
sending \(W_s\) to \(0\) and \(W_t\) to infinity. To apply Lemma \ref{lem: 2244}(b), we check that for each remaining branch point
\(e\in\{0,p,q,r\}\), the expression \((s-e)(t-e)\) is a square, and that the new leading coefficient 
\[t(r-t) (q - t) (p - t)\]
is a square as well. These statements all follow immediately from
Eq.\ \ref{eq:constant-difference}.
\end{proof}

Searching among rational points of low height on $\mathcal{S}^\circ$, we find:

\begin{theorem}\label{thm: three points on the surface}
    The three points 
 \begin{align*} 
  & P_1 = [120 : 143 : 266 : 218 : 241],\\
 & P_2 = [ 143 : 408 : 1015 : 437 : 1013], \quad\;\textrm{and}\\
 & P_3 = [16660 : 78793 : 21456 : 78644 : 27593] 
 \end{align*}
 in $\mathcal{S}^\circ(\Q)$ give pairwise
nonisomorphic curves with $\grp{2,2,2,12}\simeq
 \Jac(X_P)(\Q)_{\mathrm{tors}}$ and $J$ geometrically simple.
\end{theorem}

\begin{proof}
Magma verifies that the inclusion from Theorem \ref{thm:symmetric [2,2,2,12]} is in fact an equality in these examples. We check geometric simplicity using \cite{Lombardo2019}. 
\end{proof}

For instance, the point $P_1$ in Theorem~\ref{thm: three points on the surface}  gives the curve $$y^2 = x(x-120^2)(x-143^2)(x-266^2)(x-218^2)(x-241^2).$$  These curves appear to be the first known examples of genus-two curves with $J(\Q)_{\mathrm{tors}} \simeq [2,2,2,12]$, whether $J$ is geometrically simple or not. 

\subsection{Geometry of \texorpdfstring{$\mathcal{S}$}{\it S}}
The surface $\mathcal{S}$ has exactly 36 singular points, where the genus-two curve is maximally degenerate. They are all rational ordinary double points with coordinates in
\(\{0,1,-1\}\).  The order $192$ group 
\[
 G:= \{\pm 1\}^5/\langle-1\rangle\rtimes(S_3\times S_2)
\]
acts in the obvious way on $\mathcal{S}^\circ$, and there are 12 points in the $G$-orbit of
$[1:0:0:1:0]$ and 24 in the $G$-orbit of $[1:1:0:1:1]$.  Because
\(\mathcal S\) is a normal nodal \((2,4)\)-complete intersection, it follows from the adjunction formula that its dualizing sheaf is $\O_\mathcal{S}(1)$. In particular, $\mathcal{S}$ is of general type.

\subsection{Infinitude of \texorpdfstring{$\mathcal{S}^\circ(\Q)$}{S\textdegree(Q)}}
In the Mazur--Ogg classification of rational torsion points on elliptic curves, each group that arises does so infinitely often. Does the group $[2,2,2,12]$ appear infinitely often in the Mordell--Weil group of genus-two Jacobians over $\Q$? 

Since $\mathcal{S}$ is of general type, the Bombieri--Lang conjecture predicts that its rational points are supported on finitely many curves and points. Thus, to find infinitely many rational points, we should look for genus $0$ or $1$ curves  lying on $\mathcal{S}^\circ$. In fact, it turns out that the point $P_1 \in \mathcal{S}^\circ(\Q)$ lies on a genus $1$ curve.  We deduce this from work of Choudhry \cite{Choudhry}, where $P_1$ appears implicitly as a solution to a related Diophantine equation. 

\begin{theorem}\label{thm:g2_96_inf}
   The open surface $\mathcal{S}^\circ$ has infinitely many rational points. In particular, there are infinitely many genus-two Jacobians over $\Q$ with $J(\Q)_{\mathrm{tors}}$ containing $[2,2,2,12]$. 
\end{theorem}

\begin{proof}
    The idea is to find a hyperplane $H \subset \P^4$  such that $a^4 + b^4 + c^4 - u^4 -v^4$ factors into two quadrics modulo $a^2 + b^2 + c^2 - u^2 - v^2$ when restricted to $H$, hence $H \cap \mathcal{S}$ is a union of two $(2,2)$-curves (intersection of two quadrics) in $\P^3$. One checks that $H \, \colon \, u-v -a +b = 0$ has this property.  One of the two quadrics is singular and the corresponding component of $H \cap \mathcal{S}$ lands in the degenerate locus $\mathcal{S} \setminus \mathcal{S}^\circ$. However, the other quadric is smooth and corresponds to a non-degenerate component of $H \cap \mathcal{S}$, giving a genus $1$ curve in $\mathcal{S}^\circ$. Explicitly,  and following \cite[\S 3]{Choudhry}, we consider the elliptic curve $E \, \colon \,  y^2 = x^3  -21x -20$ over $\Q$ of rank $1$ and conductor $288 = 2^5 3^2$. Letting
    \begin{align*}
        w &= \dfrac{3x + y - 6}{y + 18}\\
        t &= \dfrac{- (4x^2 - 7x + 12y + 16)}{(x - 8) (3x + 2y + 12)},
    \end{align*}
the map $\psi \, \colon \, E \to \mathcal{S}$ given by 
\[(x,y) \mapsto [3t(1-t) : 3t-1 : 2w(3t^2+1) : 3t^2 -3t + 2 : 6t^2 - 3t + 1 ]\]
generically lands in $\mathcal{S}^\circ$. The point $Q = (-3,4)$ is a generator of the free part of $E(\Q)$, and we compute that $\psi(2Q) = P_1$. Since the image of $\psi$ is not contained in $\mathcal{S} \setminus \mathcal{S}^\circ$, it follows that infinitely many of the points $\psi(nQ)$ must lie in $\mathcal{S}^\circ$. 
\end{proof}

The proof of Theorem \ref{thm:g2_96_inf} reveals some of the arithmetic of $\mathcal{S}^\circ$ by explaining the provenance of $P_1$ and its ambient geometry. Similarly, one can ask: does the point $P_2$ (resp.\ $P_3$) lie on an irreducible curve $C \subset \mathcal{S}^\circ$ with $\#C(\Q) = \infty$?

\subsection{How we found \texorpdfstring{$\mathcal{S}^\circ$}{S\textdegree}}
We started by solving for the conditions that the sextic $f = h^2 - h_3^2$ from  Theorem \ref{thm: M(6)} splits completely. Note that $f$ already factors as $f = f_1f_2f_3$ with $\deg(f_i) = i$. The family where $f_2$ splits and $f_3$ has a rational root is easy to solve for. Forcing the remaining quadratic to split amounts to parameterizing a quadric surface over $\Q$ by $\P^2_{m,n,s}$. 

Lemma \ref{lem: 2244} then gives conditions for a given $2$-torsion point $D = W_i + W_j$ to be divisible by $2$ (and hence the existence of a point of order $4$). There are $15$ choices of $D$, each one leading to four explicit expressions $S_1,S_2,S_3,S_4$ in $m,n,s$ that are required to be squares. So each $D$ gives an explicit $(\Z/2\Z)^4$-cover $Z \to \P^2_{m,n,s}$ defined by $x_i^2 = S_i(m,n,s)$ for $1 \leq i \leq 4$.  The idea is to rationally parameterize one of the degree $4$ intermediate covers $Z_\alpha$  and then search for rational points on the resulting model of $Z$ as a biquadratic extension of $Z_\alpha$. There are 15 choices of $D$ and 35 choices of intermediate covers $Z\to Z_\alpha \to \P^2_{m,n,s}$. So there are 525 surfaces to search on. 
In theory, the choice of $D$ should not matter since they all ultimately give isomorphic $(\Z/2\Z)^4$-covers, but in practice, some choices lead to more tractable intermediate surfaces $Z_\alpha$. 

We search for rational points via the corresponding intermediate surfaces. Eventually we identify one of the two-torsion points $D$ as especially promising and find a low height solution on $Z$. We then focus on the polynomials $S_1,\ldots, S_4$ corresponding to $D$.  The condition $S_1 = x_1^2$ gives an equation of a cubic surface, which we parameterize in Magma. The subsequent equation $S_2 = x_2^2$ is a quadric surface, again parameterizable by a rational surface $\P^2_{A,B,C}$. The remaining two conditions define a surface $\mathcal{Z}$ with model
\begin{align*}
    x_3^2 &= (B^2 - C^2)(B^2-A^2)(B^2 - A^2 - C^2)\\
    x_4^2 &= (B^2 - A^2 - C^2)(B^2A^2 + B^2C^2 - A^4 -A^2C^2 - C^4),
\end{align*}
a fiber product of two degree-two K3 surfaces. A search finds two more low height points on this surface, and we eventually observe that $\mathcal{Z}$ is birational to the more symmetric surface $\mathcal{S}$ via 
\[ [A : B : C : x_3 : x_4]  =   \left[a : u : b : \frac{a^2 b^2 c}{u^2 - c^2} : v (c^2-v^2)\right].\]

\section{Other torsion groups and families}
\label{sec:other examples}

We briefly describe several other examples of genus-two Jacobians over $\Q$ with rational torsion subgroups not observed before (to our knowledge). Some of these groups were found by studying the corresponding moduli space using Lemma \ref{lem: 2244}. For the others, the strategy is similar to the previous section.  Namely, we begin working out the algebra and geometry of a moduli space of Jacobians containing that subgroup (either the full moduli space or a special locus), until either the geometry gets complicated or there are too many pathways to consider.  At each step we are trying to divide by two, three, or five using the basic arithmetic of genus-two curves (as in the proof of Theorem \ref{thm: M(6)}).  We then try to evaluate which paths seem most promising and to implement a ``smart'' brute force search, keeping track of congruence conditions, factorizations, symmetries, etc. After identifying a potentially promising pathway, we iterate this process until (hopefully) we find rational points in the search step. 

It is conjectured that $\#J(\Q)_{\mathrm{tors}}$ is uniformly bounded among genus-two Jacobians over $\Q$, so this process of finding new torsion subgroups should stop at some point. We are presumably not yet at the ``stopping point'', so we hope this paper encourages others to carry out further explorations. Even for the groups that are known to arise, there are many questions that are still open, especially the question of whether the group is realized infinitely often. See Section \ref{sec: open questions} for a collection of these questions.

One special locus of curves that proves fruitful for these searches is the locus where one of the torsion points $D -K$ is of the form $2P - K \equiv P - \iota(P)$ for some point $P$ (or the related condition that $D - K \equiv P-W$ \cite{BoxallGrant2000}). We say that such points $D-K$ {\it lie on the curve} as they are in the image of the canonical Abel-Jacobi map  $\rho \;\colon\; X \to J$ given by $\rho(P) = 2P - K$. Note that $\rho$ is not an embedding but it is injective away from the Weierstrass points. 

\subsection{A second order-80 group}

The curve
\begin{align}
 y^2+(x^2+x)y={}&-391671x^6+1894851x^5+6846924x^4\notag\\
 &-15133525x^3+3904068x^2+2625336x+254016
\label{eq:2220}
\end{align}
has  \( J(\Q)_{\mathrm{tors}} \simeq \grp{2,2,20}\) and
\(\End(J_{\Qbar})=\Z\).  This example was found in the one-parameter family
\[
y^2 = x(x-2)(x^4 + (4t + 2)x^3 + 8t(t+1)x^2 + 8t(t+1)^2x + 4(t+1)^4)
\]
of genus-two curves with both an order $5$ class $[\infty_+ - (0,0)]$ and an order $20$ class $[(-t-1,(t+1)^2(t+3)) - (0,0)]$ lying on the curve.  
The condition that the quartic factor has a root is 
\[
 t=-\frac{z^4+4z+4}{z^4+4z^3+8z^2+8z+4}.
\]
The condition that the remaining cubic has a root defines a degree-three cover $\tilde Y \to \P^1_z$ which turns out to be a singular genus-two curve. Its normalization is \href{https://www.lmfdb.org/Genus2Curve/Q/2528/a/20224/1}{the curve}
\begin{equation}\label{eq: genus two moduli space}
Y \colon y^2 +(x^2 + 1)y = x^5 - x
\end{equation}
of conductor 2528, which has a pair of rational points of large height (corresponding to $z = -1/7$ and $z = -7/9$), both corresponding to Eq. \ref{eq:2220}.   For these values we must have $[2,2,20] \hookrightarrow J(\Q)_{\mathrm{tors}}$, and we compute equality. The Jacobian of $Y$ has rank $1$, and by applying Magma's implementation of Chabauty's method at the prime $3$, we  find that these are provably the only rational points on $Y$ corresponding to non-singular genus-two curves. It follows that Eq.\ \ref{eq:2220} is the unique curve with $\grp{2,2,20}$ in this natural one-parameter family. 

As mentioned in the introduction, Elkies recently found a two-parameter family with
\(\grp{2,2,2,10}\)-torsion, parameterized 
by the Clebsch cubic surface
\(\sum r_i=\sum r_i^3=0\); see 
\cite{Elkies2024}.  The 1-parameter family above is one of several variants of Elkies' 5-torsion construction that we explored. By Elkies' result, there are infinitely many $J$ with $J(\Q)_{\mathrm{tors}}$ of order $80$, but we do not know if there are infinitely many with group $[2,2,20]$. 

\begin{remark}
    We can at least say that there are infinitely many genus-two curves over {\it quadratic} extensions $K/\Q$ such that $J(K)_{\mathrm{tors}}$ contains $[2,2,20]$. This follows from the fact that $Y$ in Eq.\ \ref{eq: genus two moduli space}  is hyperelliptic and hence has infinitely many quadratic points. In fact, since $\Jac(Y)$ has rank $1$, there are two different sources of infinitely many quadratic points on $Y$.  
\end{remark}

\subsection{\texorpdfstring{$\grp{2,2,4,4}$}{[2,2,4,4]}-torsion}

By Lemma \ref{lem: 2244}, the moduli space $A(2,2,2,4)$ of genus-two curves with $\grp{2,2,2,4}$-torsion is birational to $\P^3_{a,b,c,d}$ and the universal genus-two curve above it is
\begin{equation}
 y^2=x(x+a^2)(x+b^2)(x+c^2)(x+d^2).
\label{eq:square-branch}
\end{equation}
Lemma \ref{lem: 2244} also gives the conditions for having a second independent class of order $4$, namely:
\begin{equation}
\begin{split}
 y_1^2&=(a^2-c^2)(a^2-d^2),\\
 y_2^2&=(b^2-c^2)(b^2-d^2),\\
 w^2&=(a^2-c^2)(b^2-c^2).
\end{split}
\label{eq:2244-threefold}
\end{equation}
These three equations define a $(\Z/2\Z)^3$-cover of $\P^3_{a,b,c,d}$. Let $\mathcal{V} \, \colon \,  y^2 = (x^2 -1)(x^2 - t^2)$ be the elliptic surface over $\P^1_t$ with identity section $(1,0)$. This has Weierstrass model  \[
\mathcal{E} \colon y^2 = x(x+1)(x+u^2),
\]
where $u = (t+1)/(t-1)$. More precisely:

\begin{lemma}\label{lem: model isom}
 Let $u = (t+1)/(t-1)$. We have $\E_u \simeq \mathcal{V}_t$ via the map 
    \[(x_0,y_0) \mapsto \left( \frac{x_0-u}{x_0+u}, \frac{2(t+1)y_0}{(x_0 +u)^2}\right).\]
\end{lemma}

Taking $t = d/c$, the first two equations in Eq.\ \ref{eq:2244-threefold} describe the fiber square $\E \times_{\P^1} \E$ of this elliptic surface  over $\P^1_u$. (Note that $\E$ is the universal family over the modular curve \href{https://beta.lmfdb.org/ModularCurve/Q/4.24.0-4.b.1.3/}{$Y_1(2,4)$} by the elliptic curve analogue of Lemma \ref{lem: 2244}.)
The third equation determines a double cover $\mathcal A \to \mathcal E \times_{\P^1_u} \mathcal E$, which is related to the dual of the $2$-isogeny $\mathcal E \to \mathcal E'$ with kernel $(0,0)$ and codomain $\E' \colon y^2 = x(x-1)(x-t^2)$.  A precise description of  $\mathcal A$, which one can think of as the moduli space $A(2,2,4,4)$ of genus-two curves with $(2,2,4,4)$-torsion, is as follows.

\begin{proposition}\label{prop: 2,2,4,4 fibration of surfaces}
Let $\varphi \, \colon \, \E \to \E'$ and $\widehat{\varphi} \, \colon \, \E' \to \E$ be the above $2$-isogenies of elliptic curves over $Y = \P^1_u \setminus \{0,1,-1,\infty\}$. 
\begin{enumerate}
    \item $A(2,2,4,4)$ is isomorphic to the open subset of triples $(u, P,Q) \in \E \times_Y \E'$ such that 
    \begin{enumerate}
    \item $P$ and  $P- \widehat{\varphi}(Q)$ are not in $\grp{2,4} \simeq \E(Y)$
    \item $Q$ and $Q - \varphi(P)$ are not in $\E'[2] \simeq \E'(Y)$.
    \end{enumerate}
    \item The double cover $A(2,2,4,4) \to \E \times_Y \E$ is $(P,Q) \mapsto (P, P - \widehat{\varphi}(Q))$.
\end{enumerate}  
\end{proposition}

\begin{proof}
    In the $\mathcal{V}_t$ model for $\mathcal E$,  Eq.\ \ref{eq:2244-threefold} shows that $A(2,2,4,4)$ is the double cover of $\E \times_Y \E$ obtained by taking the square root of the function $(x_1^2 - 1)/(x_2^2 - 1)$.  For $u \in Y(\Q)$, we write $E = \E_u$ and $E' = \E'_u$ for the specializations. As in \S \ref{sec: small and two-power}, $(x,y) \mapsto x^2 -1$ is the map $E(\Q)/\widehat{\varphi}(E'(\Q)) \to \Q^\times/\Q^{\times 2}$ arising from $\widehat{\varphi}$-descent in this non-Weierstrass model. Thus, $(P_1,P_2) \in (\E \times_Y \E)(\Q)$ lifts to a rational point of $A(2,2,4,4)$ if and only if $P_1 - P_2 = \widehat\varphi(Q)$ for some $Q \in E'(\Q)$.
    Each choice of square root of $(x_1^2 - 1)/(x^2_2 - 1)$ determines a $Q$ such that $P_1 - P_2 = \widehat{\varphi}(Q)$, which proves that $A(2,2,4,4)$ is an open subset of $\E \times_Y \E'$ and proves $(2)$ as well. 
    
    That the smoothness of 
    \[X \colon y^2 = x(x+x_1^2)(x+x_2^2)(x+1^2)(x+t^2) = x(x+a^2)(x+b^2)(x+c^2)(x+d^2),\] translates to the conditions given in $(a)$ and $(b)$ follows from a tedious but straightforward computation.
\end{proof}

\begin{corollary}
    There are infinitely many genus-two curves $X/\Q$ such that $[2,2,4,4] \subset J(\Q)$. 
\end{corollary}
\begin{proof}
    It is enough to exhibit a single specialization $E = \E_u$ of positive rank over $\Q$, since then there are infinitely many pairs $P_1,P_2 \in E(\Q)$ with the same image in the finite group $E(\Q)/\widehat{\varphi}(E'(\Q))$. By construction we have $P_1 - P_2 = \widehat{\varphi}(Q)$, and so $(P_1,Q)$ gives a rational point on $\E \times_Y \E' \simeq A(2,2,4,4)$. For example, $\E_7$ has rank $1$ over $\Q$. 
\end{proof}

There are rational curves in $A(2,2,4,4)$ as well. Here is an example. 
    \begin{proposition}\label{prop: rational curve in A(2,2,4,4)}
    For $s \in \Q$, define
    \begin{align*}
    a &= (s - 2) (s^2 - s + 1/2) (s^2 - 3/2) (s^4 + s^2 + 9/4),\\
 b&=  (s^2 - 2s + 1/2) (s^2 - 2s + 9/2) (s^2 - 3/2) (s^2 + 1),\\
 c&=-(s - 2) (2s^2 - 2s + 1)(s^4 + s^2 + 9/4),\\
 d&= (s - 2) (2s^2 - 2s + 1) (2s^2 - 3) (s^2 + 1)
\end{align*}
Then whenever $X_s \colon y^2 = x(x+a^2)(x+b^2)(x+c^2)(x+d^2)$ is smooth, we have $[2,2,4,4]\subset J_s(\Q)$.
\end{proposition}

\begin{proof}
    Let $u = (1/3) (2s^2 - 1)^{-1}(2s^2 + 1)^{-1} (2s^2 - 7) (2s^2 + 3)$ and consider the elliptic surface $E_u \colon y^2 =  x(x+1)(x+u^2)$ of Mordell--Weil rank two, with independent points $P'_1$ and $P'_2$ satisfying
\begin{align*}x(P'_1) &= (-1/3) (2s^2 + 1)^{-2} (2s^2 - 1)^{-1} (2s^2 + 3) (2s^2 - 7)^2,\\
    x(P'_2) &= (1/9) s^2(2s^2 - 1)^{-2} (2s^2 - 7)^2\end{align*} and
    $$x(P'_1 +P_2') = (-1/3) (2s^2 - 1)^{-1}  (2s^2 + 3) (2s^3 - 2s^2 + 5s - 3)^{-2}  (2s^3 - 6s^2 + 5s + 3)^2.$$
    Since $x(P'_2)$ is a square, we have $P'_2 \in \widehat{\varphi}(E'(\Q))$. Let $t = (u+1)/(u-1)$ and let $P_1,P_2 \in \mathcal{V}_t(\Q) \simeq \E_u(\Q)$ be the corresponding points on $\mathcal{V}_t$.
    Then $(P_1,P_1 + P_2) \in (\E \times_Y \E)(\Q)$ lifts to a $\Q$-rational point on $A(2,2,4,4) \simeq \E \times_Y \E'$. The corresponding point on $A(2,2,4,4)$ is then $(a : b : c : d) = (x(P_1) : x(P_1 + P_2) : 1 : t)$. Using Lemma \ref{lem: model isom}, we obtain the formulas in the statement of the proposition.
\end{proof}

\begin{corollary}
    The rational points on the threefold $A(2,2,4,4)$ are Zariski dense.
\end{corollary}
\begin{proof}
 This follows from taking points of the form $(nP_s,m\varphi(P_s)) \in \E \times_Y \E'$, where $(P_s,\varphi(P_s)) \in \E \times_Y \E'(\Q)$ corresponds to the curve $X_s$ in Proposition \ref{prop: rational curve in A(2,2,4,4)}. Indeed, this gives infinitely many rational curves in $A(2,2,4,4)$ and the union of their rational points is Zariski dense in infinitely many abelian surface fibers.
 \end{proof}

The family above produced the \(\grp{2,2,4,4}\) example in
Table~\ref{tab:census}.  One can show that $A(2,2,4,4)$ is a Calabi-Yau threefold, so the Zariski density of its rational points (as opposed to a parameterization) is the best qualitative statement one can hope for. 

\subsection{[2,2,2,8]-torsion}

Let $[D - K] = (Q,\alpha)$ be the point of order $4$ in Eq.~\ref{eq:square-branch} with Mumford representation $(Q,\alpha)$ as in Lemma \ref{lem: 2244}(b).  Then $D = 2P$ for some point $P \in X(\Q)$ if and only if $\mathrm{Disc}(Q) = 0$, which is equivalent to 
\begin{equation}
S \colon  (ab+ac+ad+bc+bd+cd)^2=4abcd.
\label{eq:2228-k3}
\end{equation}
Since the class $[P - \infty]$ has order $8$, any non-degenerate $[a : b : c : d] \in S(\Q)$ corresponds to a genus-two curve with $[2,2,2,8] \hookrightarrow
J(\Q)_{\mathrm{tors}}$. One finds such rational points with a naive search. 

The quartic surface $S$ is a well-studied K3 surface, perhaps best known as the universal elliptic curve over $X_1(8)$. See \cite{BertinLecacheux} for its many incarnations and for the wealth of elliptic fibrations and rational curves on $S$.  One model for $S$ as an elliptic surface is 
\begin{equation}
    \mathcal{A}_t \colon y^2 = x(x-t-1)(x-(t+1)/t),
\end{equation}
which has the infinite order section $(1,1)$. 
Translating to the original model, we obtain the following rational curve:
\[
 a=\frac{4t^2(t+1)}{(t^2+t+1)^2},\qquad
 b=\frac{t}{t+1},\qquad c=-1,\qquad d=-t.
\]
Taking multiples in each fiber, we conclude:
\begin{theorem}
    The moduli space $M(2,2,2,8)$ over $\Q$ of genus-two Jacobians with full level $2$-structure and with an order $8$ divisor class $[P -\infty]$ is a K3 surface and has a Zariski dense set of rational points.
\end{theorem} 

\subsection{[2,22]-torsion}

We know of only two genus-two curves  with $[2,22]\hookrightarrow J(\Q)_{\mathrm{tors}}$, namely the curves
\begin{equation}\label{eq: 2-22}
    y^2 + (x^2 + x)y = x^6 - 3x^5 + 9x^4 - 5x^3 + 12x^2 - 6x
\end{equation}
and
\[y^2 + (x^2 + x)y = 30x^5 - 120x^4 - 259x^3 + 18x^2 + 138x + 36,\]
from the Booker--Sutherland database.  Interestingly, they both have $\End(J_{\bar \Q}) = \Z[\phi]$, where $\phi^2 - \phi -1 = 0$, two  Weierstrass points $W_1,W_2 \in X(\Q)$, and $P \in X(\Q)$ such that $P - W_i$ has order $22$.

\section{Special loci}\label{sec: special loci}

For a finite abelian group $G$, we informally write  $\mathcal{A}_2(G)$ for the moduli space of genus-two Jacobians together with an embedding $G \hookrightarrow J$. This is a finite cover of the Siegel modular threefold $\mathcal{A}_2 = \mathcal{A}_2(\{0\})$. 

We have seen above that for many of the largest groups $G$ realized as $J(\Q)_{\mathrm{tors}}$, the known rational points in $\mathcal{A}_2(G)$ all lie  in certain special loci of  $\mathcal{A}_2(G)$. These loci are either defined by the condition that $J$ has extra endomorphisms or by the condition that one or more of the marked torsion points lies on the curve. The fact that the largest known torsion groups in genus-two Jacobians (of order $128$) appear in split Jacobians is itself an example of this phenomenon, as $\End(J) \otimes \Q$ then contains $\Q \times \Q$. It is hard to tell if this is because we have not searched long enough for the generic examples or if this is the truth.  In any case, it would be worthwhile to study the geometry and arithmetic of these special loci on their own terms. 

Loci with extra endomorphisms define sub-Shimura varieties of $\mathcal{A}_2(G)$, or more precisely, images of other Shimura varieties under finite morphisms. The important examples here are Hilbert modular surfaces (with appropriate level structure) and their associated Humbert surfaces, which are by definition images of Hilbert modular surfaces in $\mathcal{A}_2$ (or more generally $\mathcal{A}_2(G)$) under forgetful maps. The geometry of Hilbert modular surfaces with level structure has been studied to some extent \cite{Hamahata}, but there is still much that has not been worked out, e.g.\ to classify exactly which of them are rational. Their arithmetic is even less studied, though there are a few exceptional cases \cite{AlessandriCoppola,BFS2023,Lange,LSSV}. 

Loci where one or two marked torsion divisors lie on the curve are even more mysterious. These are not in general Shimura varieties, but as the surface $\mathcal{S}^\circ$ (as well as Elkies' $[2,2,2,10]$ family) shows, they tend to have beautiful symmetries and presumably very interesting arithmetic. 

\section{Realized torsion groups}
\label{sec:census}

Table~\ref{tab:census} lists the 73 groups for which we found a genus-two curve \(X/\Q\) with \(\Jac(X)_{\Qbar}\) simple and
\(J(\Q)_{\mathrm{tors}} \simeq G\).  Each row shows
the smallest-conductor known example, preferring ``generic'' examples, i.e.\ those with \(\End(J_{\Qbar})=\Z\).  The \emph{Source(s)} column gives the first known source of a realization; rows marked ``new'' display equations first
found in the present work.
Equations are hyperlinked to an LMFDB home page in cases where one exists (many are in the ``alpha'' version of the LMFDB which contains a preliminary version of the Booker--Sutherland database).

Table~\ref{tab:splitcensus} is the geometrically split companion of
Table~\ref{tab:census}, displaying the groups for which we have found a genus-two curve $C/\Q$ with $\Jac(C)_{\Qbar}$ \emph{split} and
$J(\Q)_{\mathrm{tors}}$ equal to the displayed group. While most of these are taken from the literature, the groups \(\grp{11}\) and \(\grp{2,2,4,4}\) do not appear to have been previously realized by genus-two curves with geometrically split Jacobians.

\begin{remark}
For the groups \(\grp{2,22}\) and \(\grp{31}\) listed in Table~\ref{tab:census}, every known example is non-generic (they have real multiplication). For \(\grp{2,2,14}\) the
generic examples are new; the displayed curve is the smallest of ten.
\end{remark}

\section{Open questions}\label{sec: open questions}

For the finite abelian group $G  = [n_1,\ldots, n_4]$, we  write $A_2(G)$ for the coarse moduli space of genus-two curves $X$ together with an embedding $G \hookrightarrow \Jac(X)$. This space may have more than one connected component. 
\begin{enumerate}
    \item Find more groups $G$ that can be realized as $J(\Q)_{\mathrm{tors}}$, with $J$ geometrically simple (resp.\ split). In the geometrically simple case, the smallest unrealized group is $[5,5]$, while the smallest integer not realized as $\# J(\Q)_{\mathrm{tors}}$ is $35$. 
    \item Which groups $G$ arise as $J(\Q)_{\mathrm{tors}}$ for {\it infinitely many} pairwise geometrically non-isomorphic Jacobians $J$? 
    \item Exhibit an abelian group $G  = [n_1,n_2,n_3,n_4]$ with $n_1,n_2 \in \{1,2\}$ such that there are provably no genus-two Jacobians $J$ over $\Q$ with $J(\Q)_{\mathrm{tors}} \simeq G$.  
    (Presumably very hard.) 
    \item Formulate a conjecturally complete list of groups of the form $J(\Q)_{\mathrm{tors}}$ for some genus-two Jacobian $J$ over $\Q$. 
    \item What are the possible groups $J(\Q)_{\mathrm{tors}}$ in the geometrically split case?
    \item  What are the possible groups $J(\Q)_{\mathrm{tors}}$ in the split-over-$\Q$ case?
    \item For which groups $G$ is $A_2(G)(\Q)$ Zariski dense in $A_2(G)$? 
    \item For a given $G$, determine lower bounds on the dimension of the Zariski closure of $A_2(G)(\Q)$ in $A_2(G)$. 
    \item For a given $G$, formulate a convincing conjecture for the dimension of the Zariski closure of $A_2(G)(\Q)$ in $A_2(G)$.
    \item For a given group $G$, classify all the rational points on $A_2(G)$. 
    \item For which $G$ is there a component of $A_2(G)$ that is rational? Geometrically rational? Unirational? Uniruled? Kodaira dimension $0$? Not of general type? By \cite{Borisov}, each of these form a finite list.
    
    \item For a given $G$, is there a point in $A_2(G)(\Q)$ not contained in any curve with infinitely many rational points? Is there any  $G$ with this property?
    \item For which groups $G$ is $A_2(G)(\Q)$ supported entirely on ``special loci'', i.e.\ subvarieties of $A_2(G)$ where either $\End(J_{\overline \Q}) \neq \Z$ or $G \subset J$ intersects the image of $X \rightarrow J$ under $P \mapsto P-\iota(P)$? Are there any such $G$?
    \item Classify the Kodaira dimension of all positive dimensional special loci of $A_2(G)$ (as $G$ varies). Are there finitely many not of general type? 
    \item Given a point in $A_2(G)(\Q)$, is there a (convincing) geometric  ``explanation'' for its existence? 
\end{enumerate}

For many of these questions, even formulating a (convincing) conjectural answer would be interesting, even for a single group $G$.

\section{AI workflow and virtual lab notebook}\label{sec: ai}

This project spanned a remarkable trajectory of model capabilities.  When we started the project in July 2025, OpenAI o3, Claude Opus 4, and Gemini 2.5 Pro were the state of the art, and while we attempted to use all three (mostly out of curiosity), none of these tools contributed to our work.  Twelve months later, GPT 5.6 Sol and Fable 5 made significant contributions to our project.  The ways in which we use the tools have changed as their capabilities have improved, a coevolutionary process that we expect to continue.  Here we give our perspective, as of August 2026, on a collaboration that involved four human mathematicians and several agentic AI systems.

Aside from the capabilities trajectory, our main takeaways from the way our AI workflow evolved in this project are how impactful context is, and how powerful it is for agents to share information (with themselves, with other agents, and with humans). The latter we facilitated through a single large GitHub repository, which served as a shared lab notebook that also included skills and code that had been developed by humans and AI systems.  In the final and most productive phase of this project, we made the decision to let our agents ``off the leash'' (so to speak) and dump whatever they wanted into our shared repository with little or no human oversight. We were comfortable doing so because at that point we had essentially finished the project (or so we believed), and we figured that we had nothing to lose by letting them ``slop to their heart's content,'' given that we could easily verify any new discoveries they made using computer algebra systems.

The AI agents produced quite a lot of material (approximately 6850 files totaling nearly 300 MB), with various investigations of questionable utility.  The frontier models responsible for generating most of these files were Claude Opus 4.8 and 5, Fable 5, and GPT 5.5 and GPT 5.6 Sol, with minor contributions from Gemini 3.1 Pro and earlier Claude and GPT versions.  We have preserved these artifacts in a copy of our virtual lab notebook \cite{notebook_repo} that we have made public (this copy has been sanitized to remove API keys, details of our personal computing environments, and artifacts that are not directly relevant to this project). This repository also includes the \href{https://github.com/AndrewVSutherland/Genus2TorsionNotebook/blob/main/reports/order96-discovery-session/transcript.md}{transcript} of the conversation with Fable 5 via Claude Code in which the first \grp{2,2,2,12} example was discovered.  This is just one of many conversations that took place in the course of research, but it is a representative example of our workflow, especially in the later stages of the project.

The fact that each of us had (often extensive) histories of conversations with the models we were using meant that the agents effectively took different perspectives on the problem, and even different personas.  Rather than leading to any sort of model collapse, the interactions between the agents (via notes in our shared lab notebook) proved to be very productive. The agents also benefited from expert guidance and interaction; they often succeeded only after being pointed in the right direction, and the output they generated was in many cases curated and improved by us, and then fed back to the models. The conversation with Fable 5 that led to the discovery of $\grp{2,2,2,12}$ looks much more like a one-shot discovery than it actually was; the agent had access to several weeks' worth of discussions (both between and among humans and agents), along with summaries of several prior failed attempts aimed at this exact torsion subgroup, all of which were recorded in our lab notebook.
\bigskip\bigskip

\begingroup\scriptsize
\setlength{\LTpre}{4pt}\setlength{\LTpost}{4pt}
\setlength{\tabcolsep}{4pt}
\renewcommand{\arraystretch}{1.0}
\begin{longtable}{@{}>{\raggedright\arraybackslash}p{1.35cm}
 >{\raggedright\arraybackslash}p{\dimexpr\textwidth-1.35cm-1.3cm-8\tabcolsep\relax}
 >{\raggedright\arraybackslash}p{1.3cm}@{}}
\caption{All 73 finite abelian groups known to arise as $J(\Q)_{\mathrm{tors}}$ for a geometrically simple genus two Jacobian over \(\Q\), the
smallest-conductor known example, and sources.\label{tab:census}}\\
\toprule
Group & Curve & Source(s)\\
\midrule
\endfirsthead
\multicolumn{3}{@{}l}{{\normalsize\emph{Table~\ref{tab:census} (continued)}}}\\
\toprule
Group & Curve & Source(s)\\
\midrule
\endhead
\midrule
\endfoot
\bottomrule
\endlastfoot
$[\,]$ & \href{https://alpha.lmfdb.org/Genus2Curve/Q/?jump=[[-14580,51520,-63684,31054,-3823,-637,-18],[1]]}{$y^2 + y = -18x^{6}-637x^{5}-3823x^{4}+31054x^{3}-63684x^{2}+51520x-14580$} & \cite{BSSVY}\\
$\grp{2}$ & \href{https://www.lmfdb.org/Genus2Curve/Q/295/a/295/2}{$y^2 + (x^{2}+x+1)y = x^{5}-40x^{3}+22x^{2}+389x-608$} & \cite{BSSVY}\\
$\grp{3}$ & \href{https://www.lmfdb.org/Genus2Curve/Q/997/b/997/1}{$y^2 + y = x^{5}-2x^{4}+2x^{3}-x^{2}$} & \cite{BSSVY}\\
$\grp{4}$ & \href{https://www.lmfdb.org/Genus2Curve/Q/1070/a/2140/1}{$y^2 + (x^{3}+1)y = x^{3}-x$} & \cite{BSSVY}\\
$\grp{5}$ & \href{https://www.lmfdb.org/Genus2Curve/Q/277/a/277/2}{$y^2 + y = x^{5}-9x^{4}+14x^{3}-19x^{2}+11x-6$} & \cite{Ogg1973}\\
$\grp{6}$ & \href{https://www.lmfdb.org/Genus2Curve/Q/1038/a/1038/1}{$y^2 + (x^{2}+x)y = x^{5}-12x^{4}+26x^{3}+46x^{2}+21x+3$} & \cite{Flynn1991}\\
$\grp{7}$ & \href{https://www.lmfdb.org/Genus2Curve/Q/461/a/461/1}{$y^2 + x^{3}y = x^{5}-3x^{3}+3x-2$} & \cite{Ogg1973}\\
$\grp{8}$ & \href{https://www.lmfdb.org/Genus2Curve/Q/464/a/464/1}{$y^2 + (x+1)y = -x^{6}-2x^{5}-2x^{4}-x^{3}$} & \cite{BSSVY}\\
$\grp{9}$ & \href{https://www.lmfdb.org/Genus2Curve/Q/713/b/713/1}{$y^2 + (x^{3}+x+1)y = -x^{4}$} & \cite{Flynn1991}\\
$\grp{10}$ & \href{https://www.lmfdb.org/Genus2Curve/Q/389/a/389/1}{$y^2 + (x^{3}+x)y = x^{5}-2x^{4}-8x^{3}+16x+7$} & \cite{Flynn1991}\\
$\grp{11}$ & \href{https://www.lmfdb.org/Genus2Curve/Q/353/a/353/1}{$y^2 + (x^{3}+x+1)y = x^{2}$} & \cite{Ogg1973}\\
$\grp{12}$ & \href{https://www.lmfdb.org/Genus2Curve/Q/762/a/3048/1}{$y^2 + (x^{3}+x^{2}+x)y = x^{2}+x+1$} & \cite{BSSVY}\\
$\grp{13}$ & \href{https://www.lmfdb.org/Genus2Curve/Q/349/a/349/1}{$y^2 + (x^{3}+x^{2}+x+1)y = -x^{3}-x^{2}$} & \cite{Flynn1990,Leprevost1991a}\\
$\grp{14}$ & \href{https://www.lmfdb.org/Genus2Curve/Q/249/a/249/1}{$y^2 + (x^{3}+1)y = x^{2}+x$} & \cite{PP2012,PZP2013}\\
$\grp{15}$ & \href{https://www.lmfdb.org/Genus2Curve/Q/277/a/277/1}{$y^2 + (x^{3}+x^{2}+x+1)y = -x^{2}-x$} & \cite{Leprevost1991}\\
$\grp{16}$ & \href{https://www.lmfdb.org/Genus2Curve/Q/830/a/6640/1}{$y^2 + (x^{3}+1)y = -x^{5}+x^{4}-2x^{2}+x+1$} & \cite{BSSVY}\\
$\grp{17}$ & \href{https://www.lmfdb.org/Genus2Curve/Q/1996/b/510976/1}{$y^2 + (x^{3}+x^{2}+x)y = x^{3}+x^{2}+3x+1$} & \cite{Leprevost1991}\\
$\grp{18}$ & \href{https://www.lmfdb.org/Genus2Curve/Q/1180/a/18880/1}{$y^2 + (x^{3}+1)y = -2x^{4}+4x^{2}+2x$} & \cite{PP2012,PZP2013}\\
$\grp{19}$ & \href{https://alpha.lmfdb.org/Genus2Curve/Q/?jump=[[-2,5,-3,-5,1,1,1],[1,1,1]]}{$y^2 + (x^{2}+x+1)y = x^{6}+x^{5}+x^{4}-5x^{3}-3x^{2}+5x-2$} & \cite{Leprevost1991,PZP2013}\\
$\grp{20}$ & \href{https://www.lmfdb.org/Genus2Curve/Q/394/a/3152/1}{$y^2 + (x+1)y = -x^{5}$} & \cite{Leprevost1997}\\
$\grp{21}$ & \href{https://www.lmfdb.org/Genus2Curve/Q/388/a/776/1}{$y^2 + (x^{3}+x+1)y = -x^{4}+2x^{2}+x$} & \cite{Leprevost1991}\\
$\grp{22}$ & \href{https://www.lmfdb.org/Genus2Curve/Q/1192/a/19072/1}{$y^2 + (x^{3}+x)y = x^{3}-2x^{2}-x+1$} & \cite{Leprevost1993}\\
$\grp{23}$ & \href{https://alpha.lmfdb.org/Genus2Curve/Q/?jump=[[1,-2,2,1,-2,-1,1],[0,0,1]]}{$y^2 + x^{2}y = x^{6}-x^{5}-2x^{4}+x^{3}+2x^{2}-2x+1$} & \cite{Leprevost1993}\\
$\grp{24}$ & \href{https://www.lmfdb.org/Genus2Curve/Q/1908/a/183168/1}{$y^2 + (x^{3}+1)y = 2x^{4}+3x^{3}+4x^{2}+2x$} & \cite{Leprevost1993}\\
$\grp{25}$ & \href{https://alpha.lmfdb.org/Genus2Curve/Q/?jump=[[2,-23,41,75,25,-9],[1,1]]}{$y^2 + (x+1)y = -9x^{5}+25x^{4}+75x^{3}+41x^{2}-23x+2$} & \cite{Nicholls2018}\\
$\grp{26}$ & \href{https://alpha.lmfdb.org/Genus2Curve/Q/?jump=[[0,0,-3,-3,-3,-12],[1,1]]}{$y^2 + (x+1)y = -12x^{5}-3x^{4}-3x^{3}-3x^{2}$} & \cite{Leprevost1993}\\
$\grp{27}$ & \href{https://www.lmfdb.org/Genus2Curve/Q/604/a/9664/2}{$y^2 + (x^{3}+1)y = -x^{4}+x^{3}+x^{2}-x$} & \cite{Howe2015}\\
$\grp{28}$ & \href{https://www.lmfdb.org/Genus2Curve/Q/249/a/6723/1}{$y^2 + (x^{3}+1)y = -x^{5}+x^{3}+x^{2}+3x+2$} & \cite{PP2012,PZP2013}\\
$\grp{29}$ & \href{https://www.lmfdb.org/Genus2Curve/Q/976/a/999424/1}{$y^2 + (x+1)y = x^{6}-2x^{5}+2x^{3}-x^{2}$} & \cite{Leprevost1993}\\
$\grp{30}$ & \href{https://alpha.lmfdb.org/Genus2Curve/Q/?jump=[[0,-1,16,-12,-10,4,2],[0,0,1,1]]}{$y^2 + (x^{3}+x^{2})y = 2x^{6}+4x^{5}-10x^{4}-12x^{3}+16x^{2}-x$} & \cite{PP2015}\\
$\grp{31}$ & $y^2 + (-x^{2}-x)y = -839x^{6}+2841x^{5}-4587x^{4}+4300x^{3}-2466x^{2}+816x-126$ & \cite{costa}$^{\text{RM}}$\\
$\grp{32}$ & \href{https://alpha.lmfdb.org/Genus2Curve/Q/?jump=[[-47,-49,49,45,1,-1,1],[0,1,1]]}{$y^2 + (x^{2}+x)y = x^{6}-x^{5}+x^{4}+45x^{3}+49x^{2}-49x-47$} & \cite{Elkies2002}\\
$\grp{33}$ & \href{https://alpha.lmfdb.org/Genus2Curve/Q/?jump=[[2,13,21,7,9,-7,1],[1,1,1]]}{$y^2 + (x^{2}+x+1)y = x^{6}-7x^{5}+9x^{4}+7x^{3}+21x^{2}+13x+2$} & \cite{PP2012b,PZP2013}\\
$\grp{34}$ & \href{https://alpha.lmfdb.org/Genus2Curve/Q/?jump=[[15,12,7,3,0,-2],[0,0,1,1]]}{$y^2 + (x^{3}+x^{2})y = -2x^{5}+3x^{3}+7x^{2}+12x+15$} & \cite{Elkies2002}\\
$\grp{36}$ & \href{https://alpha.lmfdb.org/Genus2Curve/Q/?jump=[[9,6,-14,13,30,3],[0,0,1]]}{$y^2 + x^{2}y = 3x^{5}+30x^{4}+13x^{3}-14x^{2}+6x+9$} & \cite{PP2015}\\
$\grp{39}$ & \href{https://www.lmfdb.org/Genus2Curve/Q/1116/a/214272/1}{$y^2 + (x^{3}+1)y = x^{4}+2x^{3}+x^{2}-x$} & \cite{Elkies2002}\\
$\grp{40}$ & \href{https://alpha.lmfdb.org/Genus2Curve/Q/?jump=[[0,0,-2,-2,4,3],[1,1]]}{$y^2 + (x+1)y = 3x^{5}+4x^{4}-2x^{3}-2x^{2}$} & \cite{Elkies2002}\\
$\grp{2,\,2}$ & \href{https://www.lmfdb.org/Genus2Curve/Q/464/a/29696/2}{$y^2 + xy = 4x^{5}+33x^{4}+72x^{3}+16x^{2}+x$} & \cite{BSSVY}\\
$\grp{2,\,4}$ & \href{https://www.lmfdb.org/Genus2Curve/Q/997/a/997/1}{$y^2 + xy = x^{5}-8x^{4}+16x^{3}-x$} & \cite{BSSVY}\\
$\grp{2,\,6}$ & \href{https://www.lmfdb.org/Genus2Curve/Q/704/a/45056/1}{$y^2 + y = 4x^{5}+4x^{4}-x^{3}-2x^{2}$} & \cite{BSSVY}\\
$\grp{2,\,8}$ & \href{https://www.lmfdb.org/Genus2Curve/Q/464/a/29696/1}{$y^2 + (x+1)y = 8x^{5}+3x^{4}-4x^{3}-2x^{2}$} & \cite{BSSVY}\\
$\grp{2,\,10}$ & \href{https://www.lmfdb.org/Genus2Curve/Q/555/a/8325/1}{$y^2 + (x+1)y = 3x^{5}-2x^{4}-4x^{3}+x^{2}+x$} & \cite{BSSVY}\\
$\grp{2,\,12}$ & \href{https://www.lmfdb.org/Genus2Curve/Q/762/a/82296/1}{$y^2 + (x^{2}+x)y = x^{5}-8x^{4}+14x^{3}+2x^{2}-x$} & \cite{BSSVY}\\
$\grp{2,\,14}$ & \href{https://www.lmfdb.org/Genus2Curve/Q/1416/b/135936/1}{$y^2 + (x^{3}+x)y = -2x^{4}-x^{3}+x+1$} & \cite{BSSVY}\\
$\grp{2,\,16}$ & \href{https://alpha.lmfdb.org/Genus2Curve/Q/?jump=[[8,4,-3,1,-1,-1],[0,0,1,1]]}{$y^2 + (x^{3}+x^{2})y = -x^{5}-x^{4}+x^{3}-3x^{2}+4x+8$} & \cite{BookerSutherland}\\
$\grp{2,\,18}$ & \href{https://alpha.lmfdb.org/Genus2Curve/Q/?jump=[[7,2,1,3,-5,-1,1],[0,1,1]]}{$y^2 + (x^{2}+x)y = x^{6}-x^{5}-5x^{4}+3x^{3}+x^{2}+2x+7$} & \cite{BookerSutherland}\\
$\grp{2,\,20}$ & \href{https://alpha.lmfdb.org/Genus2Curve/Q/?jump=[[0,-180,-251,185,28,-30,4],[0,0,1]]}{$y^2 + x^{2}y = 4x^{6}-30x^{5}+28x^{4}+185x^{3}-251x^{2}-180x$} & \cite{BookerSutherland}\\
$\grp{2,\,22}$ & \href{https://alpha.lmfdb.org/Genus2Curve/Q/?jump=[[0,-6,12,-5,9,-3,1],[0,1,1]]}{$y^2 + (x^{2}+x)y = x^{6}-3x^{5}+9x^{4}-5x^{3}+12x^{2}-6x$} & \cite{BookerSutherland}$^{\rm RM}$\\
$\grp{2,\,26}$ & \href{https://alpha.lmfdb.org/Genus2Curve/Q/?jump=[[0,0,2,3,-12,-3,9],[1,0,1]]}{$y^2 + (x^{2}+1)y = 9x^{6}-3x^{5}-12x^{4}+3x^{3}+2x^{2}$} & \cite{BookerSutherland}\\
$\grp{2,\,28}$ & \href{https://alpha.lmfdb.org/Genus2Curve/Q/?jump=[[0,-3,4,34,-42,-3,9],[1,0,1]]}{$y^2 + (x^{2}+1)y = 9x^{6}-3x^{5}-42x^{4}+34x^{3}+4x^{2}-3x$} & \cite{BookerSutherland}\\
$\grp{3,\,3}$ & \href{https://alpha.lmfdb.org/Genus2Curve/Q/?jump=[[3,3,7,2,4,0,1],[1]]}{$y^2 + y = x^{6}+4x^{4}+2x^{3}+7x^{2}+3x+3$} & \cite{BFT2014}\\
$\grp{3,\,6}$ & \href{https://alpha.lmfdb.org/Genus2Curve/Q/?jump=[[21,-2,7,-5,0,-1],[1,1,1,1]]}{$y^2 + (x^{3}+x^{2}+x+1)y = -x^{5}-5x^{3}+7x^{2}-2x+21$} & \cite{BookerSutherland}\\
$\grp{3,\,9}$ & \href{https://alpha.lmfdb.org/Genus2Curve/Q/?jump=[[326,-75,-230,-113,40,75,16],[1,0,0,1]]}{$y^2 + (x^{3}+1)y = 16x^{6}+75x^{5}+40x^{4}-113x^{3}-230x^{2}-75x+326$} & \cite{Leprevost1995}\\
$\grp{4,\,4}$ & \href{https://alpha.lmfdb.org/Genus2Curve/Q/?jump=[[20493,6534,10832,2601,1823,242,99],[0,1,1]]}{$y^2 + (x^{2}+x)y = 99x^{6}+242x^{5}+1823x^{4}+2601x^{3}+10832x^{2}+6534x+20493$} & \cite{BookerSutherland}\\
$\grp{4,\,8}$ & \href{https://alpha.lmfdb.org/Genus2Curve/Q/?jump=[[164,682,450,-435,272,-81],[0,1,1]]}{$y^2 + (x^{2}+x)y = -81x^{5}+272x^{4}-435x^{3}+450x^{2}+682x+164$} & \cite{BookerSutherland}\\
$\grp{6,\,6}$ & $y^2 = x(39x^{2}-69x+125)(48x^{2}+8x+39)$ & new\\
$\grp{2,\,2,\,2}$ & \href{https://www.lmfdb.org/Genus2Curve/Q/2600/a/338000/1}{$y^2 + xy = 10x^{5}+8x^{4}-5x^{3}-3x^{2}+x$} & \cite{BSSVY}\\
$\grp{2,\,2,\,4}$ & \href{https://www.lmfdb.org/Genus2Curve/Q/3978/a/930852/1}{$y^2 + (x^{2}+x)y = x^{5}+3x^{4}-3x^{3}-8x^{2}+6x$} & \cite{BSSVY}\\
$\grp{2,\,2,\,6}$ & \href{https://www.lmfdb.org/Genus2Curve/Q/816/a/39168/1}{$y^2 + (x^{2}+1)y = 3x^{5}-4x^{3}-x^{2}+x$} & \cite{BSSVY}\\
$\grp{2,\,2,\,8}$ & \href{https://alpha.lmfdb.org/Genus2Curve/Q/?jump=[[0,-1,-6,-3,20,9],[0,1,1]]}{$y^2 + (x^{2}+x)y = 9x^{5}+20x^{4}-3x^{3}-6x^{2}-x$} & \cite{BookerSutherland}\\
$\grp{2,\,2,\,10}$ & \href{https://alpha.lmfdb.org/Genus2Curve/Q/?jump=[[0,0,-7,12,6,-12],[1,0,1]]}{$y^2 + (x^{2}+1)y = -12x^{5}+6x^{4}+12x^{3}-7x^{2}$} & \cite{BookerSutherland}\\
$\grp{2,\,2,\,12}$ & \href{https://alpha.lmfdb.org/Genus2Curve/Q/?jump=[[0,-5,-5,16,8,-7,1],[0,1,1]]}{$y^2 + (x^{2}+x)y = x^{6}-7x^{5}+8x^{4}+16x^{3}-5x^{2}-5x$} & \cite{BookerSutherland}\\
$\grp{2,\,2,\,14}$ & $y^2 = (x+1)(x+9)(2x-115)(6x+55)(3x^{2}-220x+2652)$ & \cite{BookerSutherland}, new\\
$\grp{2,\,2,\,20}$ & $y^2 = -(x-1)(6x+1)(2x+7)(6217x+1008)(21x^{2}-161x+144)$ & new\\
$\grp{2,\,4,\,4}$ & \href{https://alpha.lmfdb.org/Genus2Curve/Q/?jump=[[0,180,-210,-1850,2204,-116],[0,1,1]]}{$y^2 + (x^{2}+x)y = -116x^{5}+2204x^{4}-1850x^{3}-210x^{2}+180x$} & \cite{BookerSutherland}\\
$\grp{2,\,4,\,8}$ & $y^2 = x(3x-5)(5x+333)(16x^{2}+23x+576)$ & new\\
$\grp{2,\,2,\,2,\,2}$ & \href{https://alpha.lmfdb.org/Genus2Curve/Q/?jump=[[0,-1,-4,8,23,-27],[0,1,1]]}{$y^2 + (x^{2}+x)y = -27x^{5}+23x^{4}+8x^{3}-4x^{2}-x$} & \cite{BookerSutherland}\\
$\grp{2,\,2,\,2,\,4}$ & \href{https://alpha.lmfdb.org/Genus2Curve/Q/?jump=[[0,1,-6,-16,146,-225],[0,1,1]]}{$y^2 + (x^{2}+x)y = -225x^{5}+146x^{4}-16x^{3}-6x^{2}+x$} & \cite{BookerSutherland}\\
$\grp{2,\,2,\,2,\,6}$ & \href{https://alpha.lmfdb.org/Genus2Curve/Q/?jump=[[2,-12,11,18,-14,-6],[1,0,1]]}{$y^2 + (x^{2}+1)y = -6x^{5}-14x^{4}+18x^{3}+11x^{2}-12x+2$} & \cite{BookerSutherland}\\
$\grp{2,\,2,\,2,\,8}$ & $y^2 = x(x+1)(x+55^2)(x+99^2)(x+125^2)$ & new\\
$\grp{2,\,2,\,2,\,10}$ & $y^2 = x(x+1)(x-1)(3x-7)(8x-13)(24x+25)$ & \cite{Elkies2024}\\
$\grp{2,\,2,\,2,\,12}$ & $y^2 = x(x-120^2)(x-143^2)(x-266^2)(x-218^2)(x-241^2)$ & new\\
$\grp{2,\,2,\,4,\,4}$ & $y^2 = x(x+36^2)(x+57^2)(x+64^2)(x+132^2)$ & new\\
\end{longtable}
\endgroup

\bigskip\bigskip

\begingroup\scriptsize
\setlength{\LTpre}{4pt}\setlength{\LTpost}{4pt}
\setlength{\tabcolsep}{4pt}
\renewcommand{\arraystretch}{1.0}
\begin{longtable}{@{}>{\raggedright\arraybackslash}p{1.35cm}
 >{\raggedright\arraybackslash}p{\dimexpr\textwidth-1.35cm-1.3cm-8\tabcolsep\relax}
 >{\raggedright\arraybackslash}p{1.3cm}@{}}
\caption{All 77 finite abelian groups known to arise as $J(\Q)_{\mathrm{tors}}$ for a geometrically split genus two Jacobian over \(\Q\), the
smallest-conductor known example, and sources.\label{tab:splitcensus}}\\
\toprule
Group & Curve & Source(s)\\
\midrule
\endfirsthead
\multicolumn{3}{@{}l}{{\normalsize\emph{Table~\ref{tab:splitcensus} (continued)}}}\\
\toprule
Group & Curve & Source(s)\\
\midrule
\endhead
\midrule
\endfoot
\bottomrule
\endlastfoot
$[\,]$ & \href{https://www.lmfdb.org/Genus2Curve/Q/1083/b/390963/1}{$y^2 + y = -x^{6}+3x^{5}-50x^{4}+95x^{3}-14x^{2}-33x-6$} & \cite{BSSVY}\\
$\grp{2}$ & \href{https://www.lmfdb.org/Genus2Curve/Q/336/a/172032/1}{$y^2 + (x^{3}+x)y = -x^{6}+15x^{4}-75x^{2}-56$} & \cite{BSSVY}\\
$\grp{3}$ & \href{https://alpha.lmfdb.org/Genus2Curve/Q/?jump=[[-54,-318,-479,92,302,-5,-54],[1,1,0,1]]}{$y^2 + (x^{3}+x+1)y = -54x^{6}-5x^{5}+302x^{4}+92x^{3}-479x^{2}-318x-54$} & \cite{BSSVY}\\
$\grp{4}$ & \href{https://alpha.lmfdb.org/Genus2Curve/Q/?jump=[[-7536,-4824,-4239,-744,-304,76,-4],[0,0,1,1]]}{$y^2 + (x^{3}+x^{2})y = -4x^{6}+76x^{5}-304x^{4}-744x^{3}-4239x^{2}-4824x-7536$} & \cite{BSSVY}\\
$\grp{5}$ & \href{https://alpha.lmfdb.org/Genus2Curve/Q/?jump=[[-344,348,-334,17,14,-36,-8],[0]]}{$y^2 = -8x^{6}-36x^{5}+14x^{4}+17x^{3}-334x^{2}+348x-344$} & \cite{BSSVY}\\
$\grp{6}$ & \href{https://alpha.lmfdb.org/Genus2Curve/Q/?jump=[[-92,42,3,37,-12,-3,-5],[0,1,1]]}{$y^2 + (x^{2}+x)y = -5x^{6}-3x^{5}-12x^{4}+37x^{3}+3x^{2}+42x-92$} & \cite{BSSVY}\\
$\grp{7}$ & \href{https://alpha.lmfdb.org/Genus2Curve/Q/?jump=[[-3,-4,2,7,2,-4,-3],[1,0,0,1]]}{$y^2 + (x^{3}+1)y = -3x^{6}-4x^{5}+2x^{4}+7x^{3}+2x^{2}-4x-3$} & \cite{BSSVY}\\
$\grp{8}$ & \href{https://alpha.lmfdb.org/Genus2Curve/Q/?jump=[[-15,0,7,0,-4],[0,1,0,1]]}{$y^2 + (x^{3}+x)y = -4x^{4}+7x^{2}-15$} & \cite{BSSVY}\\
$\grp{9}$ & \href{https://alpha.lmfdb.org/Genus2Curve/Q/?jump=[[-3,1,-5,1,-2,-1],[0,1,0,1]]}{$y^2 + (x^{3}+x)y = -x^{5}-2x^{4}+x^{3}-5x^{2}+x-3$} & \cite{BSSVY}\\
$\grp{10}$ & \href{https://www.lmfdb.org/Genus2Curve/Q/363/a/43923/1}{$y^2 + x^{2}y = -11x^{5}-13x^{4}+7x^{3}+10x^{2}-x-2$} & \cite{BSSVY}\\
$\grp{11}$ & $y^2 = 1204142x^{6}-5109634x^{5}+31412066x^{4}-65405928x^{3}+99564424x^{2}+94188680x+9471400$ & new\\
$\grp{12}$ & \href{https://www.lmfdb.org/Genus2Curve/Q/294/a/294/1}{$y^2 + (x^{3}+1)y = x^{4}-x^{3}+x^{2}$} & \cite{BSSVY}\\
$\grp{14}$ & \href{https://alpha.lmfdb.org/Genus2Curve/Q/?jump=[[0,-4,1,-3,1,-1],[0,0,1,1]]}{$y^2 + (x^{3}+x^{2})y = -x^{5}+x^{4}-3x^{3}+x^{2}-4x$} & \cite{BookerSutherland}\\
$\grp{15}$ & \href{https://www.lmfdb.org/Genus2Curve/Q/484/a/1936/1}{$y^2 + x^{3}y = x^{4}+2x^{2}+1$} & \cite{BSSVY}\\
$\grp{16}$ & \href{https://alpha.lmfdb.org/Genus2Curve/Q/?jump=[[-125,0,97,0,-26,0,2],[0,1,0,1]]}{$y^2 + (x^{3}+x)y = 2x^{6}-26x^{4}+97x^{2}-125$} & \cite{BookerSutherland}\\
$\grp{18}$ & \href{https://alpha.lmfdb.org/Genus2Curve/Q/?jump=[[-6,5,-7,2,-1,-1],[0,0,1,1]]}{$y^2 + (x^{3}+x^{2})y = -x^{5}-x^{4}+2x^{3}-7x^{2}+5x-6$} & \cite{BookerSutherland}\\
$\grp{19}$ & \href{https://www.lmfdb.org/Genus2Curve/Q/169/a/169/1}{$y^2 + (x^{3}+x^{2}+1)y = x^{2}+x$} & \cite{MazurTate1973,Ogg1973}\\
$\grp{20}$ & \href{https://alpha.lmfdb.org/Genus2Curve/Q/?jump=[[-2,-7,0,3,-2,1],[0,1,0,1]]}{$y^2 + (x^{3}+x)y = x^{5}-2x^{4}+3x^{3}-7x-2$} & \cite{HLP2000}\\
$\grp{21}$ & \href{https://www.lmfdb.org/Genus2Curve/Q/324/a/648/1}{$y^2 + (x^{3}+x+1)y = x^{5}+2x^{4}+2x^{3}+x^{2}$} & \cite{Ogg1973}\\
$\grp{24}$ & \href{https://alpha.lmfdb.org/Genus2Curve/Q/?jump=[[375,0,97,0,8],[0,1,0,1]]}{$y^2 + (x^{3}+x)y = 8x^{4}+97x^{2}+375$} & \cite{Leprevost1995}\\
$\grp{25}$ & \href{https://alpha.lmfdb.org/Genus2Curve/Q/?jump=[[1,8,13,9,14,4,2],[0,0,1,1]]}{$y^2 + (x^{3}+x^{2})y = 2x^{6}+4x^{5}+14x^{4}+9x^{3}+13x^{2}+8x+1$} & \cite{Leprevost1995,PZP2013}\\
$\grp{27}$ & \href{https://alpha.lmfdb.org/Genus2Curve/Q/?jump=[[0,4,20,17,9,0,1],[1,1]]}{$y^2 + (x+1)y = x^{6}+9x^{4}+17x^{3}+20x^{2}+4x$} & \cite{Leprevost1995}\\
$\grp{28}$ & \href{https://alpha.lmfdb.org/Genus2Curve/Q/?jump=[[4,10,-4,-1,5,-3,1],[0,1,1]]}{$y^2 + (x^{2}+x)y = x^{6}-3x^{5}+5x^{4}-x^{3}-4x^{2}+10x+4$} & \cite{PZP2013,Howe2015}\\
$\grp{30}$ & \href{https://alpha.lmfdb.org/Genus2Curve/Q/?jump=[[2,6,10,11,10,6,2],[1,0,0,1]]}{$y^2 + (x^{3}+1)y = 2x^{6}+6x^{5}+10x^{4}+11x^{3}+10x^{2}+6x+2$} & \cite{HLP2000}\\
$\grp{35}$ & \href{https://alpha.lmfdb.org/Genus2Curve/Q/?jump=[[160,80,166,182,116,30,10],[0,1,1]]}{$y^2 + (x^{2}+x)y = 10x^{6}+30x^{5}+116x^{4}+182x^{3}+166x^{2}+80x+160$} & \cite{HLP2000}\\
$\grp{36}$ & \href{https://alpha.lmfdb.org/Genus2Curve/Q/?jump=[[9,-9,-1,-5,5,-1,1],[0,1,1]]}{$y^2 + (x^{2}+x)y = x^{6}-x^{5}+5x^{4}-5x^{3}-x^{2}-9x+9$} & \cite{PP2012b,Platonov2014}\\
$\grp{40}$ & \href{https://alpha.lmfdb.org/Genus2Curve/Q/?jump=[[66,0,12,0,2],[0,1,0,1]]}{$y^2 + (x^{3}+x)y = 2x^{4}+12x^{2}+66$} & \cite{HLP2000}\\
$\grp{45}$ & $y^2 = 13981x^{6}+29240200x^{4}+49996210000x^{2}+168300000000$ & \cite{HLP2000}\\
$\grp{48}$ & \href{https://alpha.lmfdb.org/Genus2Curve/Q/?jump=[[3,0,9,0,-4],[0,1,0,1]]}{$y^2 + (x^{3}+x)y = -4x^{4}+9x^{2}+3$} & \cite{Howe2015, PP2012b,Platonov2014}\\
$\grp{60}$ & \href{https://alpha.lmfdb.org/Genus2Curve/Q/?jump=[[1,-7,25,-24,25,-7,1],[0,1,1]]}{$y^2 + (x^{2}+x)y = x^{6}-7x^{5}+25x^{4}-24x^{3}+25x^{2}-7x+1$} & \cite{HLP2000}\\
$\grp{63}$ & $y^2 = 897x^{6}-197570x^{4}+79136353x^{2}-146398496$ & \cite{HLP2000}\\
$\grp{70}$ & $y^2 + (2x^{3}-3x^{2}-41x+110)y = x^{3}-51x^{2}+425x+179$ & \cite{Howe2015}\\
$\grp{2,\,2}$ & \href{https://alpha.lmfdb.org/Genus2Curve/Q/?jump=[[-60,-1,59,-1,59,0,-60],[1,1,1,1]]}{$y^2 + (x^{3}+x^{2}+x+1)y = -60x^{6}+59x^{4}-x^{3}+59x^{2}-x-60$} & \cite{BSSVY}\\
$\grp{2,\,4}$ & \href{https://alpha.lmfdb.org/Genus2Curve/Q/?jump=[[-16,6,-6,11,1,3],[1,0,1]]}{$y^2 + (x^{2}+1)y = 3x^{5}+x^{4}+11x^{3}-6x^{2}+6x-16$} & \cite{BSSVY}\\
$\grp{2,\,6}$ & \href{https://alpha.lmfdb.org/Genus2Curve/Q/?jump=[[-23,13,-14,-3,1,-2],[0,1,1]]}{$y^2 + (x^{2}+x)y = -2x^{5}+x^{4}-3x^{3}-14x^{2}+13x-23$} & \cite{BSSVY}\\
$\grp{2,\,8}$ & \href{https://alpha.lmfdb.org/Genus2Curve/Q/?jump=[[-214,465,-279,447,-305,60,-175],[1,0,1]]}{$y^2 + (x^{2}+1)y = -175x^{6}+60x^{5}-305x^{4}+447x^{3}-279x^{2}+465x-214$} & \cite{BSSVY}\\
$\grp{2,\,10}$ & \href{https://www.lmfdb.org/Genus2Curve/Q/256/a/512/1}{$y^2 + (x^{3}+x^{2}+x+1)y = -x^{5}-x^{4}-x^{3}-x^{2}$} & \cite{Ogg1973}\\
$\grp{2,\,12}$ & \href{https://www.lmfdb.org/Genus2Curve/Q/450/a/36450/1}{$y^2 + (x^{3}+x^{2}+x)y = 2x^{5}+4x^{4}-10x^{3}-8x^{2}+15x-5$} & \cite{HLP2000}\\
$\grp{2,\,14}$ & \href{https://alpha.lmfdb.org/Genus2Curve/Q/?jump=[[10,-6,-24,-5,-9,-3,9],[0,1,1]]}{$y^2 + (x^{2}+x)y = 9x^{6}-3x^{5}-9x^{4}-5x^{3}-24x^{2}-6x+10$} & \cite{BookerSutherland}\\
$\grp{2,\,16}$ & \href{https://alpha.lmfdb.org/Genus2Curve/Q/?jump=[[2,7,16,18,16,7,2],[1,1,1,1]]}{$y^2 + (x^{3}+x^{2}+x+1)y = 2x^{6}+7x^{5}+16x^{4}+18x^{3}+16x^{2}+7x+2$} & \cite{BookerSutherland}\\
$\grp{2,\,18}$ & \href{https://alpha.lmfdb.org/Genus2Curve/Q/?jump=[[36,66,12,-5,-3,-3,1],[0,1,1]]}{$y^2 + (x^{2}+x)y = x^{6}-3x^{5}-3x^{4}-5x^{3}+12x^{2}+66x+36$} & \cite{BookerSutherland}\\
$\grp{2,\,20}$ & \href{https://alpha.lmfdb.org/Genus2Curve/Q/?jump=[[0,-18,42,-20,15,-3,1],[0,1,1]]}{$y^2 + (x^{2}+x)y = x^{6}-3x^{5}+15x^{4}-20x^{3}+42x^{2}-18x$} & \cite{BookerSutherland}\\
$\grp{2,\,24}$ & \href{https://alpha.lmfdb.org/Genus2Curve/Q/?jump=[[15,0,-6,0,-1],[0,1,0,1]]}{$y^2 + (x^{3}+x)y = -x^{4}-6x^{2}+15$} & \cite{HLP2000}\\
$\grp{2,\,30}$ & \href{https://alpha.lmfdb.org/Genus2Curve/Q/?jump=[[12,-36,18,23,-13,-5,1],[0,1,1]]}{$y^2 + (x^{2}+x)y = x^{6}-5x^{5}-13x^{4}+23x^{3}+18x^{2}-36x+12$} & \cite{BookerSutherland}\\
$\grp{2,\,48}$ & \href{https://alpha.lmfdb.org/Genus2Curve/Q/?jump=[[1,-13,19,111,19,-13,1],[0,1,1]]}{$y^2 + (x^{2}+x)y = x^{6}-13x^{5}+19x^{4}+111x^{3}+19x^{2}-13x+1$} & \cite{BookerSutherland}\\
$\grp{3,\,3}$ & \href{https://alpha.lmfdb.org/Genus2Curve/Q/?jump=[[-2,6,-8,6,-8,6,-2],[0]]}{$y^2 = -2x^{6}+6x^{5}-8x^{4}+6x^{3}-8x^{2}+6x-2$} & \cite{BSSVY}\\
$\grp{3,\,6}$ & \href{https://alpha.lmfdb.org/Genus2Curve/Q/?jump=[[27,-27,56,-13,21,8,6],[0,1,1]]}{$y^2 + (x^{2}+x)y = 6x^{6}+8x^{5}+21x^{4}-13x^{3}+56x^{2}-27x+27$} & \cite{BSSVY}\\
$\grp{3,\,9}$ & \href{https://alpha.lmfdb.org/Genus2Curve/Q/?jump=[[0,-2,0,5,5,3,1],[1,1,1]]}{$y^2 + (x^{2}+x+1)y = x^{6}+3x^{5}+5x^{4}+5x^{3}-2x$} & \cite{HLP2000}\\
$\grp{3,\,12}$ & \href{https://alpha.lmfdb.org/Genus2Curve/Q/?jump=[[-8,12,6,-11,-3,3,1],[1,1,1]]}{$y^2 + (x^{2}+x+1)y = x^{6}+3x^{5}-3x^{4}-11x^{3}+6x^{2}+12x-8$} & \cite{HLP2000}\\
$\grp{3,\,24}$ & \href{https://alpha.lmfdb.org/Genus2Curve/Q/?jump=[[4,6,-4,-6,5,-3,1],[0,1,1]]}{$y^2 + (x^{2}+x)y = x^{6}-3x^{5}+5x^{4}-6x^{3}-4x^{2}+6x+4$} & \cite{BookerSutherland}\\
$\grp{4,\,4}$ & \href{https://alpha.lmfdb.org/Genus2Curve/Q/?jump=[[448,0,-7,0,-2,0,7],[0,0,1]]}{$y^2 + x^{2}y = 7x^{6}-2x^{4}-7x^{2}+448$} & \cite{BSSVY}\\
$\grp{4,\,8}$ & \href{https://alpha.lmfdb.org/Genus2Curve/Q/?jump=[[0,-15,-5,-5],[0,1,0,1]]}{$y^2 + (x^{3}+x)y = -5x^{3}-5x^{2}-15x$} & \cite{BookerSutherland}\\
$\grp{4,\,12}$ & \href{https://alpha.lmfdb.org/Genus2Curve/Q/?jump=[[9,-27,53,-62,53,-27,9],[0,1,1]]}{$y^2 + (x^{2}+x)y = 9x^{6}-27x^{5}+53x^{4}-62x^{3}+53x^{2}-27x+9$} & \cite{BookerSutherland}\\
$\grp{4,\,16}$ & \href{https://alpha.lmfdb.org/Genus2Curve/Q/?jump=[[252,0,-24,0,-4],[0,1,0,1]]}{$y^2 + (x^{3}+x)y = -4x^{4}-24x^{2}+252$} & \cite{BookerSutherland}\\
$\grp{5,\,5}$ & \href{https://alpha.lmfdb.org/Genus2Curve/Q/?jump=[[-2,4,2,5,2,1],[1,1,1,1]]}{$y^2 + (x^{3}+x^{2}+x+1)y = x^{5}+2x^{4}+5x^{3}+2x^{2}+4x-2$} & \cite{HLP2000}\\
$\grp{5,\,10}$ & \href{https://alpha.lmfdb.org/Genus2Curve/Q/?jump=[[-135,135,11,-33,15,-3,9],[0,1,1]]}{$y^2 + (x^{2}+x)y = 9x^{6}-3x^{5}+15x^{4}-33x^{3}+11x^{2}+135x-135$} & \cite{HLP2000}\\
$\grp{6,\,6}$ & \href{https://www.lmfdb.org/Genus2Curve/Q/196/a/21952/1}{$y^2 + (x^{2}+x)y = x^{6}+3x^{5}+6x^{4}+7x^{3}+6x^{2}+3x+1$} & \cite{HLP2000}\\
$\grp{6,\,12}$ & $y^2 = 132x^{6}+396x^{5}-6347x^{4}-13354x^{3}+75207x^{2}+81950x+88825$ & \cite{HLP2000}\\
$\grp{7,\,7}$ & $y^2 = x^{6}+3025x^{4}+3232987x^{2}+869675859$ & \cite{HLP2000}\\
$\grp{8,\,8}$ & $y^2 = 836x^{6}+88596x^{5}+88597x^{4}+1800118x^{3}-4045487x^{2}-4535664x+84285504$ & \cite{HLP2000}\\
$\grp{2,\,2,\,2}$ & \href{https://alpha.lmfdb.org/Genus2Curve/Q/?jump=[[0,1,32,268,129,16],[0,1]]}{$y^2 + xy = 16x^{5}+129x^{4}+268x^{3}+32x^{2}+x$} & \cite{BSSVY}\\
$\grp{2,\,2,\,4}$ & \href{https://alpha.lmfdb.org/Genus2Curve/Q/?jump=[[0,-15,7,18,0,3,-2],[0,1,0,1]]}{$y^2 + (x^{3}+x)y = -2x^{6}+3x^{5}+18x^{3}+7x^{2}-15x$} & \cite{BSSVY}\\
$\grp{2,\,2,\,6}$ & \href{https://www.lmfdb.org/Genus2Curve/Q/600/a/18000/1}{$y^2 + (x^{3}+x)y = 2x^{5}+3x^{4}-3x^{3}-4x^{2}+x$} & \cite{BSSVY}\\
$\grp{2,\,2,\,8}$ & \href{https://www.lmfdb.org/Genus2Curve/Q/360/a/6480/1}{$y^2 + (x+1)y = 6x^{5}-4x^{4}-5x^{3}+x^{2}+x$} & \cite{BSSVY}\\
$\grp{2,\,2,\,10}$ & \href{https://alpha.lmfdb.org/Genus2Curve/Q/?jump=[[-20,-12,66,-2,-45,3,9],[0,1,1]]}{$y^2 + (x^{2}+x)y = 9x^{6}+3x^{5}-45x^{4}-2x^{3}+66x^{2}-12x-20$} & \cite{BookerSutherland}\\
$\grp{2,\,2,\,12}$ & \href{https://alpha.lmfdb.org/Genus2Curve/Q/?jump=[[1,-9,19,19,-55,-45],[0,1,1]]}{$y^2 + (x^{2}+x)y = -45x^{5}-55x^{4}+19x^{3}+19x^{2}-9x+1$} & \cite{BookerSutherland}\\
$\grp{2,\,2,\,16}$ & \href{https://alpha.lmfdb.org/Genus2Curve/Q/?jump=[[9,-33,3,61,3,-33,-11],[0,1,1]]}{$y^2 + (x^{2}+x)y = -11x^{6}-33x^{5}+3x^{4}+61x^{3}+3x^{2}-33x+9$} & \cite{BookerSutherland}\\
$\grp{2,\,2,\,24}$ & $y^2 = x(52316x-156025)(2500x+3969)(32400x^{2}-34360x+255881)$ & \cite{HLP2000}\\
$\grp{2,\,4,\,4}$ & \href{https://alpha.lmfdb.org/Genus2Curve/Q/?jump=[[0,15,40,0,161,-240],[0,1]]}{$y^2 + xy = -240x^{5}+161x^{4}+40x^{2}+15x$} & \cite{BookerSutherland}\\
$\grp{2,\,4,\,8}$ & \href{https://alpha.lmfdb.org/Genus2Curve/Q/?jump=[[7524,-1530,-10562,797,3630,140],[0,1,1]]}{$y^2 + (x^{2}+x)y = 140x^{5}+3630x^{4}+797x^{3}-10562x^{2}-1530x+7524$} & \cite{HLP2000}\\
$\grp{2,\,6,\,6}$ & \href{https://alpha.lmfdb.org/Genus2Curve/Q/?jump=[[18,-54,9,71,-40,-5,25],[0,1,1]]}{$y^2 + (x^{2}+x)y = 25x^{6}-5x^{5}-40x^{4}+71x^{3}+9x^{2}-54x+18$} & \cite{HLP2000}\\
$\grp{2,\,2,\,2,\,2}$ & \href{https://alpha.lmfdb.org/Genus2Curve/Q/?jump=[[26,41,-44,-30,25,-4],[1,1]]}{$y^2 + (x+1)y = -4x^{5}+25x^{4}-30x^{3}-44x^{2}+41x+26$} & \cite{BookerSutherland}\\
$\grp{2,\,2,\,2,\,4}$ & \href{https://alpha.lmfdb.org/Genus2Curve/Q/?jump=[[0,-21,-6,34,16,-3],[0,1,1]]}{$y^2 + (x^{2}+x)y = -3x^{5}+16x^{4}+34x^{3}-6x^{2}-21x$} & \cite{BookerSutherland}\\
$\grp{2,\,2,\,2,\,6}$ & \href{https://alpha.lmfdb.org/Genus2Curve/Q/?jump=[[0,30,39,-72,-111,-30],[0,1,1]]}{$y^2 + (x^{2}+x)y = -30x^{5}-111x^{4}-72x^{3}+39x^{2}+30x$} & \cite{BookerSutherland}\\
$\grp{2,\,2,\,2,\,8}$ & \href{https://alpha.lmfdb.org/Genus2Curve/Q/?jump=[[-45,42,30,-23,-9,3,1],[0,1,1]]}{$y^2 + (x^{2}+x)y = x^{6}+3x^{5}-9x^{4}-23x^{3}+30x^{2}+42x-45$} & \cite{BookerSutherland}\\
$\grp{2,\,2,\,4,\,4}$ & $y^2 + (x^{2}+x)y = 60x^{5}+1000x^{4}-671x^{3}-5657x^{2}+867x+4913$ & new\\
$\grp{2,\,2,\,4,\,8}$ & $y^2 = x(x+336100^2)(x+835200^2)(x+841500^2)(x+877221^2)$ & \cite{HLP2000}\\
\end{longtable}
\endgroup

\section{Reproducibility and data availability}
\label{sec:reproducibility}

A machine-readable list of the curves and torsion subgroups listed in Tables~\ref{tab:census} and~\ref{tab:splitcensus} is available in the GitHub repository \cite{certification_repo} associated to this paper, along with Magma scripts that certify their torsion subgroups and the geometric simplicity or decomposition of their Jacobians.  These certifications rely on Stoll's algorithm for computing torsion subgroups of genus-2 Jacobians over $\Q$ \cite{Stoll1999}, using the height bounds due to M\"uller and Stoll \cite{MullerStoll2016,MullerStoll2016errata}, as implemented in Magma V2.28-9 or later \cite{Magma}, a criterion of Zywina \cite{Zywina2022} for bounding the geometric endomorphism ring of an abelian variety, which we also use to deduce geometric simplicity, and explicit isogenies or genus-one covering maps to certify geometrically split Jacobians.
This repository also contains scripts to reproduce all Magma computations referenced in this paper.


\begin{thebibliography}{99}

\bibitem{AlessandriCoppola}
J.\ Alessandrì and N.\ Coppola,
Torsion points on $\mathrm{GL}_2$-type abelian varieties,
arXiv:2602.21047 (2026).

\bibitem{certification_repo}
J. S. Balakrishnan, F. Najman, A. Shnidman, and A.V. Sutherland, Genus2Torsion GitHub repository, \url{https://github.com/AndrewVSutherland/Genus2Torsion}, 2026, [Online, accessed August 24, 2026].

\bibitem{notebook_repo}
J. S. Balakrishnan, F. Najman, A. Shnidman, and A.V. Sutherland, Genus2TorsionNotebook GitHub repository, \url{https://github.com/AndrewVSutherland/Genus2TorsionNotebook}, 2026, [Online, accessed August 24, 2026].

\bibitem{BLP2009}
N. Bernard, F. Lepr\'evost, and M. Pohst,
Jacobians of genus-2 curves with a rational point of order 11,
\emph{Experiment. Math.} \textbf{18} (2009), 65--70.

\bibitem{BertinLecacheux}
M.-J. Bertin and O. Lecacheux, Elliptic fibrations on the modular surface associated to $\Gamma_1(8)$, in {\it Arithmetic and geometry of K3 surfaces and Calabi-Yau threefolds}, 153--199, Fields Inst. Commun., 67, Springer, New York, 2013. 

\bibitem{BSSVY}
A. R. Booker, J. Sijsling, A. V. Sutherland, J. Voight, and D. Yasaki,
A database of genus-2 curves over the rational numbers,
\emph{LMS J. Comput. Math.} \textbf{19(A)} (2016), 235--254.

\bibitem{BookerSutherland}
A. R. Booker and A. V. Sutherland,
Genus 2 curves of small conductor, in preparation.

\bibitem{Borisov}
L.~A. Borisov, A finiteness theorem for subgroups of ${\rm Sp}(4,{\bf Z})$, J. Math. Sci. (New York) {\bf 94} (1999), no.~1, 1073--1099.

\bibitem{Magma}
W. Bosma, J. J. Cannon, C. Fieker, A. Steel (eds.), Handbook of Magma functions, Edition 2.29 (2026), 6510 pages.

\bibitem{BoxallGrant2000}
J. Boxall and D. Grant,
Examples of torsion points on genus two curves,
\emph{Trans. Amer. Math. Soc.} \textbf{352} (2000), 4533--4555.

\bibitem{BFS2023}
N. Bruin, E.~V. Flynn and A. Shnidman, Genus two curves with full $\sqrt{3}$-level structure and Tate-Shafarevich groups, Selecta Math. (N.S.) {\bf 29} (2023), no.~3, Paper No. 42, 33 pp.

\bibitem{BFT2014}
N. Bruin, E. V. Flynn, and D. Testa,
Descent via (3,3)-isogeny on Jacobians of genus 2 curves,
\emph{Acta Arith.} \textbf{165} (2014), 201--223.

\bibitem{Choudhry}
A. Choudhry, Equal sums of like powers with minimum number of terms, Integers {\bf 16} (2016), Paper No. A77, 11 pp.; MR3573429

\bibitem{costa}
E. Costa, N. D. Elkies, S. Hashimoto, A. Jha, K. Martin, B. Poonen, and J. Voight,
ModularAbelianSurfaces GitHub repository,
\url{https://github.com/edgarcosta/ModularAbelianSurfaces}, 2022 [Online, accessed August 24, 2026].

\bibitem{DaowsudSchmidt}
K. Daowsud and T. A. Schmidt,
Continued fractions for rational torsion,
\emph{J. Number Theory} \textbf{189} (2018), 115--130.

\bibitem{DaowsudSchmidtCorr}
K. Daowsud and T. A. Schmidt,
Corrigendum to ``Continued fractions for rational torsion,'' \emph{J. Number Theory} \textbf{246} (2023), 326--327.

\bibitem{Elkies2002}
N. D. Elkies,
Curves of genus 2 over \(\Q\) whose Jacobians are absolutely simple abelian
surfaces with torsion points of high order,
\url{https://people.math.harvard.edu/~elkies/g2_tors.html}, 2001--2002,
updated 2010.

\bibitem{Elkies2024}
N. D. Elkies,
Families of genus-2 curves with 5-torsion,
in \emph{LuCaNT: LMFDB, Computation, and Number Theory},
Contemp. Math. \textbf{796}, American Mathematical Society, 2024, 165--185.

\bibitem{EpochAI}
Epoch AI, A genus 2 curve over the rationals with a rational torsion point of prime order at least 31, \url{https://epoch.ai/frontiermath/open-problems/genus-2-jacobian-torsion}, 2026 [Online, accessed August 24, 2026].

\bibitem{Flynn1990}
E. V. Flynn,
Large rational torsion on abelian varieties,
\emph{J. Number Theory} \textbf{36} (1990), 257--265.

\bibitem{Flynn1991}
E. V. Flynn,
Sequences of rational torsions on abelian varieties,
\emph{Invent. Math.} \textbf{106} (1991), 433--442.

\bibitem{Hamahata}
Y. Hamahata, Hilbert modular surfaces with $p_g\leq 1$, Math. Nachr. {\bf 173} (1995), 193--236; MR1336961

\bibitem{HLP2000}
E. W. Howe, F. Lepr\'evost, and B. Poonen,
Large torsion subgroups of split Jacobians of curves of genus two or three,
\emph{Forum Math.} \textbf{12} (2000), 315--364.

\bibitem{Howe2015}
E. W. Howe,
Genus-2 Jacobians with torsion points of large order,
\emph{Bull. London Math. Soc.} \textbf{47} (2015), 127--135.

\bibitem{KuruSadek}
H. Kuru and M. Sadek,
Quadratic torsion orders on Jacobian varieties,
arXiv:2410.14455 (2024).

\bibitem{LSSV}
J. Laga, C.\ Schembri, A.\ Shnidman, and J.\ Voight, Rational torsion points on abelian surfaces with quaternionic multiplication, Forum Math. Sigma {\bf 12} (2024), Paper No. e92, 33 pp.

\bibitem{Lange}
H. Lange, Jacobian surfaces in ${\bf P}_4$, J. Reine Angew. Math. {\bf 372} (1986), 71--86; MR0863519

\bibitem{Leprevost1991a}
F. Lepr\'evost,
Famille de courbes de genre 2 munies d'une classe de diviseurs rationnels
d'ordre 13,
\emph{C. R. Acad. Sci. Paris S\'er. I Math.} \textbf{313} (1991), 451--454.

\bibitem{Leprevost1991}
F. Lepr\'evost,
Familles de courbes de genre 2 munies d'une classe de diviseurs rationnels
d'ordre 15, 17, 19 ou 21,
\emph{C. R. Acad. Sci. Paris S\'er. I Math.} \textbf{313} (1991), 771--774.

\bibitem{Leprevost1993}
F. Lepr\'evost,
Points rationnels de torsion de jacobiennes de certaines courbes de genre 2,
\emph{C. R. Acad. Sci. Paris S\'er. I} \textbf{316} (1993), 819--821.

\bibitem{Leprevost1995}
F. Lepr\'evost,
Jacobiennes de certaines courbes de genre 2: torsion et simplicit\'e,
\emph{J. Th\'eor. Nombres Bordeaux} \textbf{7} (1995), 283--306.

\bibitem{Leprevost1997}
F. Lepr\'evost,
Sur certains sous-groupes de torsion de jacobiennes de courbes
hyperelliptiques de genre \(g\ge 1\),
\emph{Manuscripta Math.} \textbf{92} (1997), 47--63.

\bibitem{LPS2004}
F. Lepr\'evost, M. Pohst, and A. Sch\"opp,
Rational torsion of \(J_0(N)\) for hyperelliptic modular curves and families
of Jacobians of genus 2 and genus 3 curves with a rational point of order 5,
7 or 10, \emph{Abh. Math. Sem. Univ. Hamburg} \textbf{74} (2004), 193--203.

\bibitem{LMFDB} The LMFDB Collaboration, The L-functions and modular forms database, \url{https://www.lmfdb.org}, 2026, [Online; accessed 14 August 2026].

\bibitem{Lombardo2019}
D. Lombardo,
Computing the geometric endomorphism ring of a genus-2 Jacobian,
\emph{Math. Comp.} \textbf{88} (2019), 889--929.

\bibitem{Mazur1977}
B. Mazur,
Modular curves and the Eisenstein ideal,
\emph{Publ. Math. IH\'ES} \textbf{47} (1977), 33--186.

\bibitem{MazurTate1973}
B. Mazur and J. Tate,
Points of order 13 on elliptic curves,
\emph{Invent. Math.} \textbf{22} (1973), 41--49.

\bibitem{MullerStoll2016}
J.S. M\"uller and M. Stoll,
Canonical heights on genus-2 Jacobians,
\emph{Algebra Number Theory} \textbf{10} (2016), no.~10, 2153--2234.

\bibitem{MullerStoll2016errata}
J. S. M\"uller and M. Stoll,
Errata: Canonical heights on genus-2 Jacobians,
\emph{Algebra Number Theory}, published online February 16, 2023.

\bibitem{Nicholls2018}
C. Nicholls,
\emph{Descent Methods and Torsion on Jacobians of Higher Genus Curves},
DPhil thesis, University of Oxford, 2018.

\bibitem{Ogg1973}
A. Ogg,
Rational points on certain elliptic modular curves,
in \emph{Proc. Sympos. Pure Math.} XXIV, American Mathematical Society,
1973, 221--231.

\bibitem{Ogawa1994}
H. Ogawa,
Curves of genus 2 with a rational torsion divisor of order 23,
\emph{Proc. Japan Acad. Ser. A} \textbf{70} (1994), 295--298.

\bibitem{Platonov2014} V. P. Platonov,
Number-theoretic properties of hyperelliptic fields and the torsion
problem in Jacobians of hyperelliptic curves over the rational number
field, \emph{Russian Math. Surveys} \textbf{69} (2014), no. 1, 1--34.

\bibitem{PP2012}
V. P. Platonov and M. M. Petrunin,
New orders of torsion points in Jacobians of curves of genus 2 over the
rational number field, \emph{Dokl. Math.} \textbf{85} (2012), 286--288.

\bibitem{PP2012b}
V. P. Platonov and M. M. Petrunin,
On the torsion problem in Jacobians of curves of genus 2 over the
rational number field, \emph{Dokl. Math.} \textbf{86} (2012), 642--643.

\bibitem{PP2015}
V. P. Platonov and M. M. Petrunin,
New curves of genus 2 over the field of rational numbers whose Jacobians
contain torsion points of high order,
\emph{Dokl. Math.} \textbf{91} (2015), 220--221.

\bibitem{PZP2013}
V. P. Platonov, V. S. Zhgun, and M. M. Petrunin,
On the simplicity of Jacobians for hyperelliptic curves of genus 2 over the
field of rational numbers with torsion points of high order,
\emph{Dokl. Math.} \textbf{87} (2013), 318--321.

\bibitem{Stoll1999}
M. Stoll,
On the height constant for curves of genus two,
\emph{Acta Arith.} \textbf{90} (1999), no.~2, 183--201.

\bibitem{Stoll2001}
M. Stoll,
Implementing 2-descent for Jacobians of hyperelliptic curves,
\emph{Acta Arith.} \textbf{98} (2001), no.~3, 245--277.

\bibitem{zarhin}
Y.~G. Zarhin, Division by 2 on odd degree hyperelliptic curves and their Jacobians, \emph{Izv. Math.} {\bf 83} (2019), no.~3, 501--520; translated from \emph{Izv. Ross. Akad. Nauk Ser. Mat.} {\bf 83} (2019), no.~3, 93--112.

\bibitem{Zywina2022}
D. Zywina,
Determining monodromy groups of abelian varieties,
\emph{Res. Number Theory} \textbf{8} (2022), article 89.

\end{thebibliography}
\end{document}